\documentclass{amsart}

\usepackage[mathscr]{eucal}
\usepackage{amsthm,amscd,amssymb}
\usepackage{color}
\usepackage[all]{xy}

\xyoption{arc} \xyoption{rotate}
\newcommand{\non}{\nonumber}

\newcommand{\la}{\lambda}

\newcommand{\Ga}{\Gamma}

\newcommand{\ts}{\,}
\newcommand{\tss}{\hspace{1pt}}

\newcommand{\Y}{ {\rm Y}}

\newcommand{\Z}{\mathbb{Z}\tss}

\newcommand{\gl}{\mathfrak{gl}}

\newcommand{\bal}{\begin{aligned}}
\newcommand{\eal}{\end{aligned}}
\newcommand{\beq}{\begin{equation}}
\newcommand{\ben}{\begin{equation*}}

\def\beql#1{\begin{equation}\label{#1}}

\newcommand\dcl\DeclareMathOperator

\dcl\Sp{Specm} \dcl\St{St} \dcl\cfs{cfs} \dcl\supp{supp}
\dcl\Ker{Ker} \dcl\Hom{Hom} \dcl\Ext{Ext} \dcl\Ann{Ann} \dcl\Ob{Ob}
\dcl\im{Im} \dcl\mo{mod} \dcl\rank{rank} \dcl\M{M} \dcl\Specm{Specm}
\dcl\Aut{Aut} \dcl\lm{lm} \dcl\lc{lc} \dcl\lt{lt} \dcl\Gal{Gal}

\newcommand{\bm}{\mathbf m}

\newcounter{numberofremark}
\dcl\gkdim{GKdim}

\newcommand\de{{\delta}}

\newtheorem{theorem}{Theorem}[section]
\newtheorem{definition}[theorem]{Definition}
\newtheorem{corollary}[theorem]{Corollary}
\newtheorem{lemma}[theorem]{Lemma}
\newtheorem{proposition}[theorem]{Proposition}
\newtheorem{example}[theorem]{Example}
\newtheorem{remark}[theorem]{Remark}

\newcommand{\cond}[2]{
\ifthenelse {\equal{\commenttag} {detailed}} {#2$\maltese$}
 {#1}
 }

\dcl\Der{Der} \dcl\Inn{Inn} \dcl\mcd{mcd} \dcl\GK{GK} \dcl\Ass{Ass}
\dcl\Spec{Spec}

\newcommand{\eeq}{\end{equation}}
\newcommand{\een}{\end{equation*}}

\def\bm{\mathbf m}

\usepackage{color}

\begin{document}

\title[RELATION MODULES    FOR FINITE $W$-ALGEBRAS]{RELATION MODULES FOR FINITE $W$-ALGEBRAS AND TENSOR PRODUCTS OF HIGHEST WEIGHT EVALUATION MODULES FOR YANGIANS OF TYPE A}
\author{Vyacheslav Futorny}
\address{Institute of Mathematics and Statistics\\
University of S\~ao Paulo\\
Caixa Postal 66281- CEP 05315-970\\
S\~ao Paulo, Brazil} \email{futorny@ime.usp.br}

\author{Luis Enrique Ramirez}
\address{Universidade Federal do ABC, Santo Andr\'e SP, Brasil} \email{luis.enrique@ufabc.edu.br}

\author{Jian Zhang}
\address{Institute of Mathematics and Statistics\\
University of S\~ao Paulo\\
Caixa Postal 66281- CEP 05315-970\\
S\~ao Paulo, Brazil} \email{zhang@ime.usp.br}

\begin{abstract}
We construct explicitly a large family of Gelfand-Tsetlin modules for an arbitrary finite $W$-algebra of type $A$ and establish their irreducibility. A basis of these modules is formed by the Gelfand-Tsetlin tableaux whose entries satisfy certain admissible sets of relations. Characterization and an effective method of constructing such admissible relations are given. In the case of the Yangian of $\mathfrak{gl}_n$ we prove the sufficient condition for the irreducibility of the tensor product of two highest weight relation modules and establish irreducibility of any number of highest weight relation modules with generic highest weights. This extends
the results of Molev to infinite dimensional highest modules.

\end{abstract}

\vspace{0,2cm}

\maketitle

\section {Introduction}
$W$-algebras were first introduced in the work of Zamolodchikov in the 80's  in the study  of two-dimensional  conformal field theories.
General definition of $W$-algebras was given
in  the work of Feigin and
Frenkel \cite{FF}  via quantized Drinfeld-Sokolov
reduction. This was later generalized by
Kac, Roan and Wakimoto \cite{KRW}, Kac and
Wakimoto \cite{KW} and  De Sole and Kac \cite{SK}. For  basic representation theory of $W$-algebras we refer to \cite{A1} and \cite{A2}.

$W$-algebras can be viewed as affinizations of finite $W$-algebras.  A finite $W$-algebra
is associated to a simple complex finite-dimensional Lie algebra and
to its nilpotent elements. Their concept goes back to the  papers
of Kostant \cite{Ko}, Lynch \cite{L}, Elashvili and Kac \cite{EK}.
Finite $W$-algebras are related to quantizations of the Slodowy slices \cite{P}, \cite{GG} and to the Yangian theory \cite{RS}, \cite{BK1}. In type $A$, that is for $\gl_n$,
Brundan and Kleshchev \cite{BK1}, \cite{BK2} showed  that finite $W$-algebras are isomorphic to
certain quotients of the shifted Yangians.

If $\pi=\pi(p_1,\dots, p_n)$ is a pyramid with $N=p_1+\cdots+p_{n}$  boxes distributed in $n$ rows with  $p_1,\dots, p_n$ boxes in each row respectively (counting from the bottom),
then the
finite $W$-algebra $W(\pi)$ is associated with $\gl_N$ and the
nilpotent matrix in $\gl_N$ of Jordan type $(p_1,\dots, p_n)$. In particular, $W(\pi)$ is the universal enveloping algebra of  $\gl_n$ if the pyramid $\pi$ has one column with $n$ boxes.

Theory of Gelfand-Tsetlin representations for finite $W$-algebras of type $A$ was developed in \cite{FMO2}. In such representations the Gelfand-Tsetlin subalgebra of $W(\pi)$ has a common generalized eigenspace decomposition. For an irreducible representation this is equivalent to require the existence of a common eigenvector for the Gelfand-Tsetlin subalgebra $\Gamma$. Such an eigenvector is annihilated by some maximal ideal of $\Gamma$.
The main problem is to construct explicitly (with a basis and the action of algebra generators) irreducible  Gelfand-Tsetlin modules for $W(\pi)$ generated by a vector  annihilated by a fixed maximal ideal of $\Gamma$.
Recent results of \cite{FGRZ} allow to construct a ``universal" cyclic Gelfand-Tsetlin module for $W(\pi)$ for a fixed maximal ideal of $\Gamma$. When this module is irreducible (sufficient condition is given in \cite{FGRZ}) the problem of explicit construction is solved. On the other hand,  even for $\gl_n$ not all irreducible subquotients of the universal module have a tableaux basis.   Hence, this difficult problem of explicit construction  of irreducible  Gelfand-Tsetlin modules remains open.

A new  technique of constructing certain irreducible  Gelfand-Tsetlin modules  was developed in  \cite{FRZ} in the case of the universal enveloping algebra of  $\gl_n$ generalizing the work of Gelfand and Graev \cite{GeG} and the work of Lemire and Patera \cite{LP}.   The main objective of this paper is to adapt and apply the  technique of \cite{FRZ}  in
the case of finite $W$-algebras of type $A$.  We obtain:\\

\noindent - Effective  removal of relations  method (the RR-method) for constructing  admissible sets of relations (Theorem \ref{rr}); \\
- Characterization of  admissible sets of relations (Theorem \ref{sufficiency of admissible});\\
- Explicit construction of  Gelfand-Tsetlin $W(\pi)$-modules  for a given admissible set of relations (Definition \ref{def-main}).  \\

Our main result is the following:

\begin{theorem}\label{thm-B}
For a given admissible set of relations $\mathcal{C}$ and any tableau $[l]$ satisfying $\mathcal{C}$, the space $V_{\mathcal{C}}([l])$ (see Definition \ref{def-sets}) is a Gelfand-Tsetlin $W(\pi)$-module with diagonal action of the Gelfand-Tsetlin subalgebra.
\end{theorem}

As a consequence we construct a large new family of   Gelfand-Tsetlin $W(\pi)$-modules with explicit basis and  the action the generators of algebra.
If $\mathcal{C}$ is an admissible set of relations  and  $[l]$  is any tableau  satisfying $\mathcal{C}$, then we have a Gelfand-Tsetlin
 $W(\pi)$-module $V_{\mathcal{C}}([l])$ which we call the \emph{relation module} associated with $\mathcal{C}$ and $[l]$.
 We have the following criterion of irreducibility for the relation modules (Theorem \ref{thm-irr}):

\begin{theorem}\label{thm-A}
The Gelfand-Tsetlin module $V_{\mathcal{C}}([l])$ is irreducible if and only if $\mathcal{C}$ is the maximal admissible set of relations satisfied by $[l]$.
\end{theorem}

Finally, we consider a tensor product of relation modules.
If $V_1, \ldots, V_l$ are $\gl_n$-modules then $V_1\otimes \ldots \otimes V_l$ is a module for the Yangian $\Y(\gl_n)$.  For finite dimensional
$\gl_n$-modules the criterion of irreducibility of such tensor product was established in \cite{M1}. We consider  tensor product of infinite dimensional  highest weight relation modules.
 We establish irreducibility of any number of highest weight relation modules with \emph{generic} highest weights (Theorem \ref{thm-generic}) and give sufficient conditions of the irreducibility of two highest weight relation modules  (Theorem \ref{thm-integral}).
These results
 extend the results of Molev \cite{M1} and Brundan and Kleshchev \cite{BK2} to infinite dimensional highest weight modules for the Yangians.
 We observe that we do not fully cover the above mentioned results since not all finite dimensional $\Y(\gl_n)$-modules
 are relation modules.

\section{Finite $W$-algebras
}\label{sec:sy}\setcounter{equation}{0}
The ground field will be the field of complex numbers $\mathbb{C}$.

Fix a tuple $(p_1,\dots,p_n)$ such that $1\leqslant p_1\leqslant \dots\leqslant p_n$.
Associate with this tuple
the pyramid $\pi=\pi(p_1,\dots,p_n)$, where $p_i$ is the number of unit squares in the $i$th row of the
pyramid counting from the bottom.   We will assume that the rows of $\pi$ are left-justified.
From now on we set  $N=p_1+\cdots+p_{n}$.

 Given such pyramid  $\pi$,
the corresponding {\it shifted Yangian\/} $\Y_{\pi}(\gl_n)$ \cite{BK1}
is the associative algebra over $\mathbb{C}$ defined by generators
\begin{align}\label{gener}
d^{\ts(r)}_i,&\quad i=1,\dots,n,&&\quad r\geqslant 1,\\
f^{(r)}_i,&\quad i=1,\dots,n-1,&&\quad r\geqslant 1,
\non\\
e^{(r)}_i,&\quad i=1,\dots,n-1,&&\quad r\geqslant p_{i+1}-p_i+1,
\non
\end{align}
subject to the following relations:
\begin{align}
[d_{i}^{\ts(r)},d_{j}^{\ts(s)}]&=0,
\non\\
[e_{i}^{(r)},f_{j}^{(s)}]&=-\ts \de_{ij}\ts\sum_{t=0}^{r+s-1}
d_{i}^{\tss\prime\ts(t)}\ts d_{i+1}^{\ts(r+s-t-1)},\non\\
[d_{i}^{\ts(r)},e_{j}^{(s)}]&=(\de_{ij}-\de_{i,j+1})\ts\sum_{t=0}^{r-1}
d_{i}^{\ts(t)}\ts e_{j}^{(r+s-t-1)},\non\\
[d_{i}^{\ts(r)},f_{j}^{(s)}]&=
(\de_{i,j+1}-\de_{ij})\ts\sum_{t=0}^{r-1}
f_{j}^{(r+s-t-1)}\ts d_{i}^{\ts(t)},
\non
\end{align}
\begin{align*}
[e_{i}^{(r)},e_{i}^{(s+1)}]-[e_{i}^{(r+1)},e_{i}^{(s)}]&=
e_{i}^{(r)}e_{i}^{(s)}+e_{i}^{(s)}e_{i}^{(r)},\non\\
[f_{i}^{(r+1)},f_{i}^{(s)}]-[f_{i}^{(r)},f_{i}^{(s+1)}]&=
f_{i}^{(r)} f_{i}^{(s)}+f_{i}^{(s)} f_{i}^{(r)},\non\\
[e_{i}^{(r)},e_{i+1}^{(s+1)}]-[e_{i}^{(r+1)},e_{i+1}^{(s)}]&=
-e_{i}^{(r)}e_{i+1}^{(s)},\non\\
[f_{i}^{(r+1)},f_{i+1}^{(s)}]-[f_{i}^{(r)},f_{i+1}^{(s+1)}]&=
-f_{i+1}^{(s)}f_{i}^{(r)},
\non
\end{align*}
\begin{alignat}{2}
[e_{i}^{(r)},e_{j}^{(s)}]&=0,\qquad&&\text{if}\quad |i-j|>1,\non\\
[f_{i}^{(r)},f_{j}^{(s)}]&=0,\qquad&&\text{if}\quad |i-j|>1,\non\\
[e_{i}^{(r)},[e_{i}^{(s)},e_{j}^{(t)}]]
&+[e_{i}^{(s)},[e_{i}^{(r)},e_{j}^{(t)}]]=0,
\qquad&&\text{if}\quad |i-j|=1,\non\\
[f_{i}^{(r)},[f_{i}^{(s)},f_{j}^{(t)}]]
&+[f_{i}^{(s)},[f_{i}^{(r)},f_{j}^{(t)}]]=0,
\qquad&&\text{if}\quad |i-j|=1,
\non
\end{alignat}
for all possible $i,j,r,s,t$, where $d_{i}^{(0)}=1$ and
the elements
$d_{i}^{\tss\prime\ts(r)}$ are obtained from the relations

\begin{equation*}
\sum_{t=0}^r d_{i}^{\tss(t)}\ts d_{i}^{\tss\prime\ts(r-t)}=\de_{r0},
\qquad r=0,1,\dots.
\end{equation*}

Note that the algebra $\Y_{\pi}(\gl_n)$ depends only on
the differences $p_{i+1}-p_i$ (see (\ref{gener})), and our definition
corresponds to the left-justified
pyramid $\pi$, as compared to \cite{BK1}.
In the case
of a rectangular pyramid $\pi$ with $p_1=\dots=p_n$, the
algebra $\Y_{\pi}(\gl_n)$ is isomorphic to the {\it Yangian\/}
$\Y(\gl_n)$; cf. \cite{m:yc}.
Moreover, for an arbitrary pyramid $\pi$,
the shifted Yangian $\Y_{\pi}(\gl_n)$
can be regarded as a natural subalgebra of $\Y(\gl_n)$.

Following \cite{BK1},
the {\it finite $W$-algebra\/} $W(\pi)$,
associated with $\gl_N$ and the pyramid $\pi$, can be defined
as the quotient of $\Y_{\pi}(\gl_n)$ by the two-sided ideal
generated by all elements $d_{1}^{\tss(r)}$ with $r\geqslant p_1+1$.  In the case of the
one-column pyramids $\pi$ we obtain
 the universal enveloping algebra of $\gl_n$.
We refer the reader to \cite{BK1, BK2} for a description
 and the structure of the algebra $W(\pi)$,
 including an analog of the Poincar\'e--Birkhoff--Witt theorem
as well as a construction of algebraically independent
generators of the center of $W(\pi)$.

\subsection{Gelfand-Tsetlin modules}
Recall that the pyramid $\pi$ has left-justified
rows $(p_1,\dots, p_n)$. Denote $\pi_k$   the pyramid associated with the tuple $(p_1,\dots, p_k)$,
 and let $W(\pi_k)$ be the corresponding finite $W$-algebra, $k=1,\dots, n$. Then
 we have the following
chain of subalgebras
\beql{chainw}
W(\pi_1)\subset
W(\pi_2) \subset\dots\subset W(\pi_n)=W(\pi).
\eeq
Denote by
$\Gamma$ the commutative subalgebra of $W(\pi)$ generated by the
centers of the subalgebras $W(\pi_k)$ for $k=1,\dots, n$, which is the \textit{Gelfand--Tsetlin subalgebra}
of $W(\pi)$ \cite{BK2}.

A finitely generated module $M$ over $W(\pi)$ is called a {\em
Gelfand-Tsetlin module\/} (with respect to $\Gamma$) if \begin{equation*}
M=\underset{{\bm} \in \Specm {\Gamma}}{\bigoplus}M({\bm}) \end{equation*} as
a $\Ga$-module, where \begin{equation*} M({\bm}) \ = \ \{ x\in M\ | \ {\bm}^k x
=0\quad \text{for some}\quad k\geqslant 0\} \end{equation*}
 and $\Specm \Ga$
denotes the set of maximal ideals of $\Ga$.

Theory of Gelfand-Tsetlin modules for $W(\pi)$ was developed in \cite{FMO1}, \cite{FMO2}, \cite{FMO3}.
In particular, it was shown

\begin{theorem}\label{thm-finite}[\cite{FMO3}, Theorem II]
Given any $\bm\in
\Specm \Gamma$ the number $F(n)$ of non-isomorphic irreducible Gelfand-Tsetlin modules $M$ over
$W(\pi)$ with $M(\bm)\neq 0$ is non-empty and finite.
\end{theorem}

The proof of this result is based on the important fact that the finite $W$-algebra $W(\pi)$ is a Galois order \cite{FO} (or equivalently, integral Galois algebra) (\cite{FMO3}, Theorem 3.6). Moreover, in  particular cases of
one-column pyramids \cite{O}  and two-row  pyramids
\cite{FMO3},  the number $F(n)$ is bounded by  $p_1!(p_1+p_2)!\ldots (p_1+\ldots+
p_{n-1})!$. This remains a conjecture in general.

\
\

\subsection{Finite-dimensional
representations of $W(\pi)$}
Set
\ben
\bal
f_i(u)=\sum_{r=1}^{\infty} f_i^{(r)}\ts u^{-r}, \qquad
e_i(u)&=\sum_{r=p_{i+1}-p_i+1}^{\infty} e_i^{(r)}\ts u^{-r}
\eal
\een
and denote
\ben
A_i(u)=u^{p_1}\ts (u-1)^{p_2}\ts\dots (u-i+1)^{p_i}\ts
a_i(u)
\een
for $i=1,\dots,n$ with $a_i(u)=d_1(u)\ts d_2(u-1)\dots d_i(u-i+1)$,
and
\ben
\bal
B_i(u)&=u^{p_1}\ts (u-1)^{p_2}\ts
\dots (u-i+2)^{p_{i-1}}\ts (u-i+1)^{p_{i+1}}\ts a_i(u)
\ts e_i(u-i+1),\\
C_i(u)&=u^{p_1}\ts (u-1)^{p_2}\ts\dots (u-i+1)^{p_i}
\ts f_i(u-i+1)\ts a_i(u)
\eal
\een
for $i=1,\dots,n-1$. Then
$A_i(u)$, $B_i(u)$, and $C_i(u)$, $i=1,\ldots,n$
are polynomials in $u$, and their coefficients
are generators of $W(\pi)$ \cite{FMO2}. Define
the elements $a_{r}^{(k)}$ for $r=1,\dots,n$
and $k=1,\dots,p_1+\dots+p_r$ through the expansion
\ben
A_r(u)=u^{p_1+\dots+p_r}+\sum_{k=1}^{p_1+\dots+p_r}
a_{r}^{(k)}\ts u^{p_1+\dots+p_r-k}.
\een
Thus, the elements $a_{r}^{(k)}$ generate the Gelfand--Tsetlin
subalgebra $\Gamma$ of $W(\pi)$.

Fix an $n$-tuple
$\la(u)=\big(\la_1(u),\dots,\la_n(u)\big)$
of monic polynomials in $u$, where $\la_i(u)$
has degree $p_i$. Let $L(\la(u))$ denote
the irreducible highest weight representation
of $W(\pi)$ with highest weight $\la(u)$.
Then $L(\la(u))$ is a Gelfand-Tsetlin module generated by
a nonzero vector $\xi$
such that
\begin{alignat}{2}
B_{i}(u)\ts\xi&=0 \qquad &&\text{for} \quad
i=1,\dots,n-1, \qquad \text{and}
\non\\
u^{p_i}\ts d_{i}(u)\ts\xi&=\la_i(u)\ts\xi \qquad &&\text{for}
\quad i=1,\dots,n.
\non
\end{alignat}
Let \ben \la_i(u)=(u+\la^{(1)}_i)\ts(u+\la^{(2)}_i) \dots
(u+\la^{(p_i)}_i),\qquad i=1,\dots,n. \een We assume that
the parameters $\la^{(k)}_i$ satisfy the  conditions
\ben
\la^{(k)}_i-\la^{(k)}_{i+1}\in\Z_{\geq 0},\qquad i=1,\dots,n-1, \een
for any value
$k\in\{1,\dots,p_i\}$.
In this case, the
representation $L(\la(u))$ of $W(\pi)$ is finite-dimensional.

The explicit construction of a family of
finite-dimensional irreducible representations of
$W(\pi)$ was given in \cite{FMO2}. As it will play an important role in the arguments  of this paper. We recall below this
construction.

\subsection{Gelfand-Tsetlin basis for finite-dimensional
representations}
Consider a family of finite-dimensional representations of $W(\pi)$
by imposing the condition \ben
\la_i^{(k)}-\la_j^{(m)}\notin\Z,\qquad\text{for all}\ \ i,j
\quad\text{and all}\ \  k\ne m \een
on a highest weight $\lambda(u)$.
 The {\it standard Gelfand--Tsetlin
tableau\/} $\mu(u)$ associated with the highest weight $\la(u)$ is
an array of rows $(\la_{r1}(u),\dots,\la_{rr}(u))$ of monic
polynomials in $u$ for $r=1,\dots,n$, where \ben
\la_{ri}(u)=(u+\la_{ri}^{(1)})\dots(u+\la_{ri}^{(p_i)}), \qquad
1\leqslant i\leqslant r\leqslant n, \end{equation*}
 with
$\la_{ni}^{(k)}=\la_{i}^{(k)}$, such that the top row coincides with
$\la(u)$, and \ben
\la_{r+1,i}^{(k)}-\la_{ri}^{(k)}\in\Z_{\geq 0}\qquad\text{and}\qquad
\la_{ri}^{(k)}-\la_{r+1,i+1}^{(k)}\in\Z_{\geq 0} \een for $k=1,\dots,p_i$
and $1\leqslant i\leqslant r\leqslant n-1$.

The following result was shown in \cite{FMO2}.

\begin{theorem}\label{thm:dgactylp} The representation $L(\la(u))$ of the
algebra $W(\pi)$ allows a basis $\{\xi_{\mu}\}$ parametrized by all standard tableaux $\mu(u)$ associated with $\la(u)$ such that the action of
the generators is given by the formulas \beql{aractba} A_r(u)\ts
\xi^{}_{\mu}=\la_{r1}(u)\dots\la_{rr}(u-r+1)\ts \xi^{}_{\mu}, \eeq
for $r=1,\dots,n$, and
\begin{align}\label{bractba}
B_r(-l^{(k)}_{ri})\ts \xi^{}_{\mu}&=
-\la_{r+1,1}(-l^{\tss(k)}_{ri})\dots\la_{r+1,r+1}(-l^{\tss(k)}_{ri}-r)
\ts \xi^{}_{\mu+\de_{ri}^{(k)}},\\
C_r(-l^{(k)}_{ri})\ts \xi^{}_{\mu}&=
\la_{r-1,1}(-l^{\tss(k)}_{ri})\dots\la_{r-1,r-1}(-l^{\tss(k)}_{ri}-r+2)
\ts \xi^{}_{\mu-\de_{ri}^{(k)}}, \non
\end{align}
for $r=1,\dots,n-1$, where $l^{\tss(k)}_{ri}=\la^{(k)}_{ri}-i+1$ and $\xi^{}_{\mu\pm\de_{ri}^{(k)}}$
corresponds to the tableau obtained from $\mu(u)$ by replacing
$\la_{ri}^{(k)}$ by $\la_{ri}^{(k)}\pm 1$, while the vector
$\xi^{}_{\mu}$ is set to be zero if $\mu(u)$ is not a standard tableau. \end{theorem}

The action of the operators
$B_r(u)$ and $C_r(u)$ for an arbitrary value of
$u$ can be calculated using the Lagrange
interpolation formula.

For convenience we denote $\xi_{\mu}$ by $[l]$. and set

$$ l_{ri}(u)=(u+l_{ri}^{(1)})\dots(u+l_{ri}^{(p_i)}), \qquad 1\leqslant i\leqslant r\leqslant n, $$
with $l_{ni}^{(k)}=\la_{i}^{(k)}-i+1$, which implies
\begin{equation} \label{standard relations}
l_{r+1,i}^{(k)}-l_{ri}^{(k)}\in\Z_{\geq 0}\qquad\text{and}\qquad
l_{ri}^{(k)}-l_{r+1,i+1}^{(k)}\in\Z_{>0}
\end{equation}
 for $k=1,\dots,p_i$
and $1\leqslant i\leqslant r\leqslant n-1$.

Then $\lambda_{ri}(u)=\l_{ri}(u+i-1)$.
The Gelfand-Tsetlin formulas can be rewritten  as follows:
\beql{aractba} A_r(u) [l]=l_{r1}(u)\dots l_{rr}(u)\ts [l], \eeq
for $r=1,\dots,n$, and
\begin{align}\label{bractba}
B_r(-l^{(k)}_{ri})\ts [l]&=
-l_{r+1,1}(-l^{\tss(k)}_{ri})\dots l_{r+1,r+1}(-l^{\tss(k)}_{ri} )
\ts  [l+\de_{ri}^{(k)}] ,\\
C_r(-l^{(k)}_{ri})\ts [l]&=
l_{r-1,1}(-l^{\tss(k)}_{ri})\dots l_{r-1,r-1}(-l^{\tss(k)}_{ri})
\ts  [l-\de_{ri}^{(k)}] , \non
\end{align}
for $r=1,\dots,n-1$, where $ [l\pm\de_{ri}^{(k)}] $
corresponds to the tableau obtained from $[l]$ by replacing
$l_{ri}^{(k)}$ by $l_{ri}^{(k)}\pm 1$, while the vector
$[l]$ is set to be zero if it does not satisfies \eqref{standard relations}.

By the Lagrange interpolation formula we have
\begin{equation}\label{action abc}
\begin{split}
A_r(u) [l]&=l_{r1}(u)\dots l_{rr}(u)\ts [l],\\
B_r(u)\ts [l] &=-\sum_{i,k}\left(
\frac{\prod\limits_{j,t}(l_{r+1,j}^{(t)}-l_{r,i}^{(k)})}{\prod\limits_{(j,t)\neq (i,k)}(l_{r,j}^{(t)}-l_{r,i}^{(k)})}\right)
\prod\limits_{(j,t)\neq (i,k)}(u+l_{r,j}^{(t)})  \ts [l+\de_{ri}^{(k)}] ,\\
C_r(u)\ts [l] &=
\sum_{i,k} \left( \frac{\prod\limits_{j,t}(l_{r-1,j}^{(t)}-l_{r,i}^{(k)})}{\prod\limits_{(j,t)\neq (i,k)}(l_{r,j}^{(t)}-l_{r,i}^{(k)})}\right)
\prod\limits_{(j,t)\neq (i,k)}(u+l_{r,j}^{(t)})
\ts [l-\de_{ri}^{(k)}].
\end{split}
\end{equation}

It is easy to see that
$d_{r}(u)=a_{r-1}^{-1}(u)a_r(u)=(u-r+1)^{-p_r}A_{r-1}^{-1}A_{r}$.
Then the action of $d_r(u)$ is given by
\begin{equation*}
d_{r}(u)[l]=\frac{l_{r1}(u)\dots l_{rr}(u)}{(u-r+1)^{p_r}l_{r-1,1}(u)\dots l_{r-1,r-1}(u)}[l].
\end{equation*}

Note that the polynomials $l_{r1}(u)\cdots l_{rr}(u)$ and $(u-r+1)^{p_r}l_{r-1,1}(u)\cdots l_{r-1,r-1}(u)$
have the same degree $p_1+\cdots+p_r$. Hence $\frac{l_{r1}(u)\dots l_{rr}(u)}{(u-r+1)^{p_r}l_{r-1,1}(u)\dots l_{r-1,r-1}(u)}$ can be written as the following formal series in $u$:
\ben
\bal
1+\sum_{t=1}^{\infty} d_r^{\ts(t)}(l)\ts u^{-t},
\eal
\een
where $d_r^{\ts(t)}(l)$ is a polynomial in $l_{r,i}^{(k)}$ and $l_{r-1,j}^{(s)}$ with $1\leq i \leq r$, $1\leq k \leq p_i$, $1\leq j \leq r-1$, $1\leq k \leq p_j$.
Thus $d_{r}(u)[l]=d_r^{\ts(t)}(l)[l]$.

Similarly, since
\begin{equation*}
\begin{split}
e_r(u)=u^{p_r-p_{r+1}}A_r^{-1}(u+r-1)B_r(u+r-1),\\
f_r(u)=C_r(u+r-1) A_r^{-1}(u+r-1),
\end{split}
\end{equation*}
the action of $e_r$ and $f_r$ is given by

\begin{equation}\label{action bc}
\begin{split}
e_r(u)\ts [l] &=-\sum_{i,k}\left(
\frac{\prod\limits_{j,t}(l_{r+1,j}^{(t)}-l_{r,i}^{(k)})}{\prod\limits_{(j,t)\neq (i,k)}(l_{r,j}^{(t)}-l_{r,i}^{(k)})}
\frac{\prod\limits_{(j,t)\neq (i,k)}(u+r-1+l_{r,j}^{(t)})}{u^{p_{r+1}-p_r}\prod\limits_{(j,t) }(u+r-1+l_{r,j}^{(t)}+\de_{ri}^{(k)})}\right) \ts [l+\de_{ri}^{(k)}] ,\\
f_r(u)\ts [l] &=
\sum_{i,k}\left(  \frac{\prod\limits_{j,t}(l_{r-1,j}^{(t)}-l_{r,i}^{(k)})}{\prod\limits_{(j,t)\neq (i,k)}(l_{r,j}^{(t)}-l_{r,i}^{(k)})}
\frac{\prod\limits_{(j,t)\neq (i,k)}(u+r-1+l_{r,j}^{(t)})}{\prod\limits_{(j,t) }(u+r-1+l_{r,j}^{(t)})}\right)
\ts [l-\de_{ri}^{(k)}].
\end{split}
\end{equation}

Since $\prod\limits_{(j,t)\neq (i,k)}(u+r-1+l_{r,j}^{(t)})$ is a polynomial in $u$ of degree $p_1+\cdots+p_r-1$ while
$\prod\limits_{(j,t)}(u+r-1+l_{r,j}^{(t)}+\de_{ri}^{(k)})$ and $\prod\limits_{(j,t)}(u+r-1+l_{r,j}^{(t)})$ are polynomials of degree $p_1+\cdots+p_r$, we can write the two rational functions in \eqref{action bc} as follows:
\begin{align*}
\frac{\prod\limits_{(j,t)\neq (i,k)}(u+r-1+l_{r,j}^{(t)})}{u^{p_{r+1}-p_r}\prod\limits_{(j,t) }
(u+r-1+l_{r,j}^{(t)}+\de_{ri}^{(k)})}
&=\sum_{t=p_{r+1}-p_r+1}^{\infty} b_{r,k,i}^{(t)}(l)\ts u^{-t},\\
\frac{\prod\limits_{(j,t)\neq (i,k)}(u+r-1+l_{r,j}^{(t)})}{\prod\limits_{(j,t) }(u+r-1+l_{r,j}^{(t)})}
&=\sum_{t=1}^{\infty} c_{r,k,i}^{(t)}(l)\ts u^{-t},
\end{align*}
where $b_{r,k,i}^{(t)}(l)$ and $c_{r,k,i}^{(t)}(l)$ are polynomials in $l_{r,i}^{(k)}$  with $1\leq i \leq r$, $1\leq k \leq p_i$  and
$b_{r,k,i}^{(p_{r+1}-p_r+1)}(l)=c_{r,k,i}^{(1)}(l)=1$.  Therefore the action of $e_r^{(t)}$ and $f_r^{(t)}$ can be expressed as follows:
\begin{equation}\label{action of generators}
\begin{split}
e_r^{(t)}\ts [l] &=-\sum_{i,k}\left(
\frac{\prod\limits_{j,t}(l_{r+1,j}^{(t)}-l_{r,i}^{(k)})}{\prod\limits_{(j,t)\neq (i,k)}(l_{r,j}^{(t)}-l_{r,i}^{(k)})}
b_{r,k,i}^{(t)}(l)\right)  \ts [l+\de_{ri}^{(k)}] ,\\
f_r^{(t)}\ts [l] &=
\sum_{i,k}\left(  \frac{\prod\limits_{j,t}(l_{r-1,j}^{(t)}-l_{r,i}^{(k)})}{\prod\limits_{(j,t)\neq (i,k)}(l_{r,j}^{(t)}-l_{r,i}^{(k)})}
c_{r,k,i}^{(t)}(l)\right) \ts [l-\de_{ri}^{(k)}].
\end{split}
\end{equation}

\section{Admissible sets of relations}
 In this section we discuss admissible sets of relations and obtain their characterization.  Each such set defines an infinite family of Gelfand-Tsetlin modules over $W(\pi)$.

Let $a,b\in\mathbb{C}$, from now on whenever we write $a\geq b\ (\text{respectively }a>b)$ we will mean $a-b\in\mathbb{Z}_{\geq 0}$ (respectively $a-b\in\mathbb{Z}_{>0}$).\\
Set $\mathfrak{V}:=\{(k,i,j)\ |\ 1\leq j\leq i\leq n,\ 1\leq k \leq p_j\}$. For triples $(k,i,j)$, $(r,s,t)$ from $\mathfrak{V}$  we say that $[l]$ satisfies the relation $(k,i,j)\geq (r,s,t)$ (respectively $(k,i,j)>(r,s,t)$)  if $l_{ij}^{(k)}\geq l_{st}^{(r)}\ (\text{respectively }l_{ij}^{(k)}> l_{st}^{(r)})$.

From now on when we write a triple $(k,i,j)$  we assume that $1\leq j\leq i\leq n,1 \leq k \leq p_j$ without mentioning this restriction.

\begin{definition}A subset of relations
$$\{ (k_1,i_{1},j_1)\geq (\mathrm{or} >) \ (k_2,i_2,j_2), \cdots,\  (k_m,i_{m},j_m)\geq(\mathrm{or} >)\ (k_1,i_1,j_1)\}$$  will be called a \emph{loop}.
\end{definition}

Set $\mathcal{R}=\mathcal{R}_1\cup \mathcal{R}_2$, where
\begin{align*}
\mathcal{R}_1&:=\{(k,i,j)\geq(k',i-1,j'),\ (r,s-1,t)>(r',s,t')  |\ 2\leq i,s \leq n  \}\\
\mathcal{R}_{2}&:=\{(k,n,i)\geq(k',n,j)\ |\ 1\leq i\neq j\leq n\}.
\end{align*}

From now on we will consider sets of relations $\mathcal{C}$ which are subsets of  $\mathcal{R}$ and always assume that they do not contain a loop
in the top row, that is  the sets of relations do  not contain any subset of the form $\{ (k_1,n,j_1)\geq  (k_2,n,j_2), \cdots,\  (k_m,n,j_m)\geq  (k_1,n,j_1)\}$.

Given $\mathcal{C}$, denote by $\mathfrak{V}(\mathcal{C})$ the set of all triples $(k,i,j)$ in $\mathfrak{V}$  involved in some relation of $\mathcal{C}$.

Let $\mathcal{C}_1$ and $\mathcal{C}_2$ be two subsets of $\mathcal{C}$.
We say that $\mathcal{C}_1$ and $\mathcal{C}_2$ are disconnected
if
$\mathfrak{V}(\mathcal{C}_1)\cap\mathfrak{V}(\mathcal{C}_2)=\emptyset$,
 otherwise  we say that $\mathcal{C}_1$ and $\mathcal{C}_2$ are connected.
 A subset  $\mathcal{C}\subseteq\mathcal{R}$ is called \emph{decomposable} if it can be decomposed into the
  union of two disconnected subsets of $\mathcal{R}$, otherwise $\mathcal{C}$ is called \emph{indecomposable}. Clearly,
any subset of $\mathcal{R}$ is  a union of disconnected  indecomposable sets, moreover, such decomposition is unique.


\begin{definition}\label{def-sets} Let $\mathcal{C}$ be a subset of $\mathcal{R}$, and
$[l]$ a Gelfand-Tsetlin tableau.
\begin{itemize}
\item[1]
We say that $[l]$ \emph{satisfies $\mathcal{C}$} if $[l]$ satisfies all relations in $\mathcal{C}$ and $l_{ki}^{(s)}-l_{kj}^{(t)}\in \mathbb{Z}$ only if they are in the same indecomposable subset.
\item[2.] We call a tableau $[l]$ \emph{noncritical} if $l_{ki}^{(s)}\neq l_{kj}^{(t)}$ for all $(s,i)\neq (t,j)$,  otherwise the tableau $[l]$ is \emph{critical}.
\item[3.] We call  $\mathcal{C}$ \emph{noncritical} if any $[l]$ satisfying $\mathcal{C}$ is noncritical.
\item[4.] Suppose $[l]$ satisfy $\mathcal{C}$.
Denote by $\mathcal{B}_{\mathcal{C}}([l])$   the set of all tableaux $[l+z]$ satisfying $\mathcal{C}$,
with $z_{ij}^{(k)}\in \mathbb Z $, $z_{nj}^{(k)}=0$ and by $V_{\mathcal{C}}([l])$ the complex vector space
spanned by $\mathcal{B}_{\mathcal{C}}([l])$.
\item[5.] Let $\mathcal{C}_1,\mathcal{C}_2$ be noncritical sets of relations. We say that $\mathcal{C}_{1}$ implies $\mathcal{C}_{2}$ if any tableau that satisfies $\mathcal{C}_{1}$ also satisfies $\mathcal{C}_{2}$. We say that $\mathcal{C}_{1}$ is equivalent to $C_{2}$ if $\mathcal{C}_{1}$ implies $\mathcal{C}_{2}$ and $\mathcal{C}_{2}$ implies $\mathcal{C}_{1}$.
\end{itemize}
\end{definition}
Set $\mathcal{S}:=\{(k,i+1,j)\geq(k,i,j)>(k,i+1,j+1)\ |\   i\leq n-1 \}$. We say that a tableau $[l]$ is \emph{standard} if and only if $[l]$ satisfies all the relations in $\mathcal{S}$
 and $l_{ij}^{(k)}-l_{ij}^{(r)}\notin \mathbb{Z}$ for any $k\neq r$.

Using Definition \ref{def-sets}  the basis in Theorem \ref{thm:dgactylp} containing a standard Gelfand-Tsetlin tableau $[l]$ can be described as  $\mathcal{B}_{\mathcal{S}}([l])$.


\begin{definition}\label{def-main}
Let $\mathcal{C}$ be any subset of $\mathcal{R}$. We call $\mathcal{C}$  \emph{admissible} if for any $[l]$ satisfying  $\mathcal{C}$,  the vector space $V_{\mathcal{C}}([l])$  has a structure of a $W(\pi)$-module, endowed with the action of $W(\pi)$ given by the formulas \eqref{action abc}.
\end{definition}

It follows from Theorem \ref{thm:dgactylp} that $\mathcal{S}$ is admissible. Hence, our goal is to determine admissible sets of relations.

Let $\mathcal{C}$ be a subset of $\mathcal{R}$ and $(k,i,j)\in \mathfrak{V}(\mathcal{C})$. We call $(k,i,j)$ \emph{maximal} (with respect to  $\mathcal{C}$) if there is no
$(r,s,t)\in \mathfrak{V}(\mathcal{C})$ such that $(r,s,t)\geq(\text{or} > ) \ (k,i,j)$. Minimal triples are defined similarly.

Description of admissible sets is a difficult problem. Nevertheless, the {\it relations removal method} ({\it RR-method} for short), developed in \cite{FRZ} can be applied
in the case of finite $W$-algebras and provides
 an effective tool of constructing admissible subsets of $\mathcal{R}$.

Let $\mathcal{C}$ be any admissible subset of $\mathcal{R}$ and  $(k,i,j)\in \mathfrak{V}(\mathcal{C})$ be maximal or minimal.
Denote by $\mathcal{C}\setminus \{(k,i,j)\}$ the set of relations obtained from $\mathcal{C}$ by removing all relations that involve $(k,i,j)$. We say that $\widetilde{\mathcal{C}}\subseteq\mathcal{C}$ is obtained from $\mathcal{C}$ by the RR-method if $\widetilde{\mathcal{C}}$ is obtained from $\mathcal{C}$ by a sequence of such  removings of relations for different indexes. That is, there exist $\{(k_1,i_1,j_1),\ldots,(k_t,i_t,j_t)\}\subseteq\mathfrak{V}(\mathcal{C})$ such that $((k_{r+1},i_{r+1},j_{r+1}))$ is maximal or minimal with respect to $\mathcal{C}\setminus \{(k_1,i_1,j_1)\}\setminus \{(k_2,i_2,j_2)\}\cdots \setminus \{(k_r,i_r,j_r)\}$ and $\widetilde{\mathcal{C}}=\mathcal{C}\setminus \{(k_1,i_1,j_1)\}\setminus \{(k_2,i_2,j_2)\}\cdots \setminus \{(k_t,i_t,j_t)\}$.

Let $\Omega_n$ be the free abelian group generated by the Kronecker delta's $\delta^{(k)}_{ij}$,
$1\leq j \leq i\leq n-1$, $1\leq k \leq p_j$. We can identify $\Omega_n$ with the set of integral tableaux with zero top rows.

\begin{theorem}\label{rr}
Let $ \mathcal{C}_1$ be any admissible subset of $\mathcal{R} $ and suppose that $\mathcal{C}_2$ is obtained from $ \mathcal{C}_1$ by the RR-method,  then $\mathcal{C}_2$ is admissible.
\end{theorem}

\begin{proof}
Suppose $\mathcal{C}_2$ is obtained from $ \mathcal{C}_1$ by removing the relations involving $(k,i,j)$. To show $\mathcal{C}_2$ is admissible it is sufficient to prove that for any $[l]$ satisfying $\mathcal{C}_2$
and any defining relation  $g=0$  of $W(\pi)$ we have $g[l]=0$. The proof of this fact generalizes the argument of the proof of Theorem 4.7 in \cite{FRZ}.

Assume $(k, i, j)$ to be maximal (resp. minimal)
and $m$ some positive (resp. negative) integer  with $|m|>3$. Let $[\gamma]$ be a tableau
satisfying the relations  $\mathcal{C}_1$ and
$\gamma_{st}^{(r)}=l_{st}^{(r)}$ if $(s, t, r)\neq (k,i, j)$.
Then $[\gamma+ m\delta^{(k)}_{ij}]$  satisfies $ \mathcal{C}_1$ and  $V_{\mathcal{C}_{1}}([\gamma +m\delta^{(k)}_{ij}])$ is a $W(\pi)$-module.
 We have
\begin{align*}
g([\gamma+ m\delta^{(k)}_{ij}])&=\sum\limits_{w\in A}g_{w}(\gamma+ m\delta^{(k)}_{ij})[\gamma+ m\delta^{(k)}_{ij} + w],\\
\end{align*}
where $A\subset \Omega_n$ is such that $[\gamma + m\delta^{(k)}_{ij} +w] $ are nonzero in the corresponding formulas.
One has that $[l+w]=0 $  in $V_{\mathcal C_2}([l])$ if and only if $[\gamma+z+w]=0$ in $V_{\mathcal C_1}([\gamma+m\delta^{(l)}_{ij}])$ when $|m|>3$.
Thus,
\begin{align*}
g[l]&=\sum_{w \in A }g_{w}([l])([l+w]).
\end{align*}
Since $V_{\mathcal{C}_1}([\gamma+m\delta^{(l)}_{ij}])$ is a module for infinitely many values of $m$ and
$g_{w}([\gamma+ m\delta^{(l)}_{ij}])$  are rational functions in the variable $m$,
we conclude that $g_{w}([l+w])=0$ for all $w\in A$ and, hence,  $ \mathcal{C}_2 $  is admissible.
\end{proof}

Since empty set can be obtained from $ \mathcal{S}$ applying the RR-method finitely many times, Theorem \ref{rr} immediately implies:

\begin{corollary}
Empty set is admissible. In particular, let $[l]$ be a tableau with $\mathbb Z$-independent entries (i.e. the differences of entries on the same row are non-integers) $l_{ij}^{k}$, $1\leq j\leq i\leq n,1 \leq k \leq p_j$,
$\mathcal{B} ([l])$   the set of all tableaux $[l+z]$
with $z_{ij}^{(k)}\in \mathbb Z $, $z_{nj}^{(k)}=0$ and  $V ([l])$ the complex vector space
spanned by $\mathcal{B} ([l])$. Then $V ([l])$ is a   $W(\pi)$-module  with the action of generators given by the formulas \eqref{action abc}.
\end{corollary}

If we denote by $R_{i}$ the number of entries of the form $l_{ij}^{(k)}$ on a tableau. We have a natural action of $G:=S_{R_{1}}\times\cdots\times S_{R_{n}}$ on Gelfand-Tsetlin tableaux by permutation of elements of the same row of the tableau.
Since the Gelfand-Tsetlin formulas \eqref{action abc} are $G$-invariant, we immediately have:

\begin{lemma}\label{lemma permutation}
Fix $i$. Let $\sigma$ be a permutation of the set $\{(k,i,j),1\leq j\leq i ,1 \leq k \leq p_j\}$. If
$ \mathcal{C}$ is admissible then $ \sigma\mathcal{C}$ is admissible.
\end{lemma}

To visualize the relations we will draw an arrow down-right to indicate the relation $\geq$ and an arrow up-right to indicate the relation $>$.

\begin{example}\label{example admissible}
The following sets are admissible by  Theorem \ref{thm:dgactylp} and Theorem \ref{rr}.
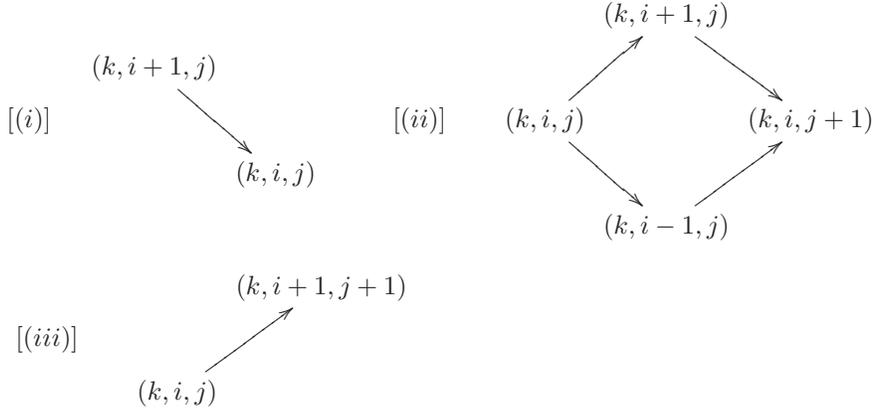
$$ [(i)]\ \
\begin{tabular}{c c }
\xymatrix @C=0.1em {
 (k,i+1,j) \ar[rd]& &    &    \\
  & (k,i,j)    & &    \\
}
\end{tabular}
\ \ [(ii)]\ \
\begin{tabular}{c c }
\xymatrix @C=0.1em {
   & &  (k,i+1,j)\ar[rd]   & &   \\
  & (k,i,j)  \ar[ru]\ar[rd]   & &(k,i,j+1)   \\
    &     &(k,i-1,j)\ar[ru]  &   \\
}
\end{tabular}$$

$
[(iii)]\ \
\begin{tabular}{c c }
\xymatrix @C=0.1em {
 & &  (k,i+1,j+1)   &    \\
  & (k,i,j) \ar[ru]    & &    \\
}
\end{tabular}
$\\

It follows from Lemma \ref{lemma permutation} that the permutations of these sets are also admissible.
\end{example}

\begin{example}\label{example non admissilbe}
The following sets are not admissible.

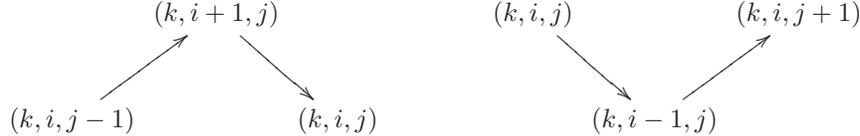
\begin{tabular}{c c}
\xymatrix @C=0.1em {
 &(k,i+1,j)\ar[rd]   &   \\
 (k,i,j-1)  \ar[ru]  &    &(k,i,j)   \\
}

&\qquad \ \ \xymatrix@C=0.1em{
 (k,i,j)\ar[rd]   &    &(k,i,j+1)   \\
 &(k,i-1,j)\ar[ru]   &   \\
}\\
\end{tabular}

 Hence, the permutations of these sets are not admissible either by Lemma \ref{lemma permutation}.
\end{example}

\begin{definition}
Let $\mathcal{C}$ be an indecomposable noncritical subset of $\mathcal{R} $. A subset of $\mathcal{C}$ of the form $\{(k,i,j)> (k,i+1,t),\ (k,i+1,s)\geq(k,i,r)\}$    with $j<r$ and $s<t$ will be called a \emph{cross}.
\end{definition}

\begin{proposition}
Let $\mathcal{C}$ be an indecomposable noncritical subset of $\mathcal{R} $.
If $\mathcal{C}$ contains a cross, then it is not admissible.
\end{proposition}

\begin{proof}
Indeed, assume that $\mathcal{C}$ is admissible and contains a cross. Then,  applying the RR-method to $\mathcal{C}$ we will obtain a set of relations from Example \ref{example non admissilbe} (see details  in  \cite{FRZ}). Therefore $\mathcal{C}$ is not admissible.
\end{proof}




\begin{definition}\label{def reduced}
Let $\mathcal{C}$ be any noncritical set of relations. We call  $\mathcal{C}$ \emph{reduced} if for every $(k,i,j)\in \mathfrak{V}(\mathcal{C})$ the following conditions
are satisfied:
\begin{itemize}
\item[(i)] There exist at most one $(r,t)$ such that $(r,i+1,t)\geq(k,i,j)\in \mathcal{C}$,
\item[(ii)] There exist at most one $(r,t)$ such that $(k,i,j)>(r,i+1,t)\in\mathcal{C}$,
\item[(iii)] There exist at most one $(r,t)$ such that $(k,i,j)\geq(r,i-1,t)\in\mathcal{C}$,
\item[(iv)] There exist at most one $(r,t)$ such that $(r,i-1,t)>(k,i,j)\in \mathcal{C}$,
\item[(v)] Any relation in the top row is not implied by other relations.

\end{itemize}
\end{definition}

The following important result follows from \cite{FRZ}, Theorem 4.17.
\begin{theorem}\label{equiv to reduced set}
Any noncritical set of relations is equivalent to an unique reduced set of relations.
\end{theorem}

\begin{definition}
Let $\mathcal{C}$ be any subset of $\mathcal{R}$. Given $(k,i,j),\ (r,s,t)\in \mathfrak{V}(\mathcal{C})$ we will write:
\begin{itemize}
\item[(i)] $(k,i,j)\succeq_{\mathcal{C}} (r,s,t)$ if, there exists $\{(k_1,i_{1},j_{1}),\ldots,(k_m,i_{m},j_{m})\}\subseteq \mathfrak{V}(\mathcal{C})$ such that
\begin{align}\label{geq sub C}
\{(k,i,j)\geq (k_1,i_{1},j_1), \cdots,\  (k_m,i_{m},j_m)\geq (r,s,t)\}&\subseteq \mathcal{C}
\end{align}
\item[(ii)] We write $(k,i,j)\succ_{\mathcal{C}} (r,s,t)$ if there exists $\{(k_1,i_{1},j_{1}),\ldots,(k_m,i_{m},j_{m})\}\subseteq \mathfrak{V}(\mathcal{C})$ satisfying (\ref{geq sub C}), with one of the inequalities being $>$.
\end{itemize}
\end{definition}

Let $\mathcal{C}$ be an indecomposable set and $\prec$ be the lexicographical order. We say that $\mathcal{C} $ is {\it pre-admissible}  if the following conditions are satisfied:
\begin{itemize}
\item[(i)] $\mathcal{C} $ is noncritical;
\item[(ii)] $(k, i,j)\succ_{\mathcal{C}}(r,i,t)$ and only if $(k,j)\prec(r,t)$ for $i=1,\ldots,n-1$;
\item[(iii)] $(k, n,j)\succeq_{\mathcal{C}}(r,n,t)$ and only if $(k,j)\prec(r,t)$ for $i=1,\ldots,n-1$;
\item[(iv)] $\mathcal{C} $ does not contain crosses.
\end{itemize}

An arbitrary set $\mathcal{C}$ is pre-admissible if every  indecomposable subset of $\mathcal{C}$ is pre-admissible.
From now on we will only consider pre-admissible sets, since any admissible set is pre-admissible (see \cite{FRZ} Section $4$).
Denote by $\mathfrak{F}$ the set of all indecomposable sets $\mathcal{C}$ which satisfy the following condition:
for every adjoining triples $(k,i,j)$ and $(r,i,s)$, $1\leq i\leq n-1$,
there exist $(k_1,j_1)$ and $(k_2,j_2)$ satisfying one of the following conditions:

\begin{equation}\label{condition for admissible}
\begin{split}
 &\{(k,i,j)>(k_1,i+1,j_1)\geq (r,i,s),\ (k,i,j)\geq (k_2,i,j_2)>(r,i,s)\}\subseteq \mathcal{C},\\
 &\{(k,i,j)>(k_1,i+1,j_1),(k_2,i+1,j_2)\geq (r,i,s)\}\subseteq  \mathcal{C}, (k_1,j_1)\prec(k_2,j_2).\\
 \end{split}
\end{equation}

The main result of this section is the following theorem which gives a characterization of admissible sets of relations. A detailed proof will be given in Section $\S 6$. For the universal enveloping algebra of
$\gl_n$ this result was established in \cite{FRZ}, Theorem 4.27.

\

\begin{theorem}\label{sufficiency of admissible} A pre-admissible set of relations  $\mathcal{C}$ is admissible if and only if
$\mathcal{C}$ is a union of indecomposable sets from $\mathfrak{F}$.
\end{theorem}

For an admissible set of relations $\mathcal{C}$ and any $[l]$ which satisfies $\mathcal{C}$, the $W(\pi)$-module $V_{\mathcal{C}}([l])$ is a Gelfand-Tsetlin module.
We will call it a \emph{relation module}.

\section{irreducibility of relation modules}

In this section we establish the criterion of irreducibility of the relation module $V_{\mathcal{C}}([l])$.
We say that a set
$\mathcal{C}$ is the  {\it maximal} set of relations for $[l]$ if $[l]$ satisfies $\mathcal{C}$ and if $[l]$ satisfies a set of relations $\mathcal{C}'$, then $\mathcal{C}$ implies $\mathcal{C}'$.



\begin{lemma}\label{lemma separation}
Let $\sum_{\mu}c_{\mu}[l_{\mu}]$ be a vector in $V_{\mathcal{C}}([l])$ and nonzero $c_{\mu}$.
Then   $[l_{\mu}]\in V_{\mathcal{C}}([l])$ for each $\mu$.
\end{lemma}

\begin{proof}
Suppose $c_{\mu}[l_{\mu}]+c_{\nu}[l_{\nu}]\in V_{\mathcal{C}}([l])$ and $[l_{\mu}]$, $[l_{\nu}]$ have different entries in $r$-th row. By Theorem \ref{thm:dgactylp} we have
$ A_r(u)\ts [l^{}_{\mu}]= a_{\mu} [l^{}_{\mu}]$, $ A_r(u)\ts [l^{}_{\mu}]= a_{\nu} [l^{}_{\mu}]$. Moreover, $a_{\mu}= a_{\nu}$ if and only if
the $r$-th row of $[l_{\nu}]$ is a permutation of that of $[l_{\mu}]$ which is a contradiction with the non criticality of $\mathcal{C}$. So $a_{\mu}\neq a_{\nu}$ and
 both $[l_{\mu}]$ and $[l_{\nu}] $ are in $V_{\mathcal{C}}(l)$. The general case can easily be  seen by induction on the number of terms in the linear combination.
\end{proof}

\begin{lemma}\label{lemma sequences}
Let $\mathcal{C}$ be an admissible set of relations,
 $[l]$ and $[\gamma]$ be tableaux satisfying $\mathcal{C}$ and $l^{(k)}_{n,i}=\xi^{(k)}_{n,i}$, $l^{(k)}_{r,i}-\xi^{(k)}_{r,i}\in \mathbb{Z}, 1\leq r \leq n-1$ . Then there exist $\{(k_t,i_t,j_t)\}_{t=1,\ldots,s}\subseteq \mathfrak{V}(\mathcal{C})$
such that for any $r\leq s$, $[l+\sum_{t=1}^{r}\epsilon_t\delta^{(k_t)}_{i_t,j_t}]$ satisfies  $\mathcal{C}$
and $[l+\sum_{t=1}^{s}\epsilon_t\delta^{(k_t)}_{i_t,j_t}]=[\gamma]$,
where
$\epsilon_t= 1$ if $\xi^{(k_t)}_{i_t,j_t}-l^{(k_t)}_{i_t,j_t}\geq 0$ and
$\epsilon_t= - 1$ if $\xi^{(k_t)}_{i_t,j_t}-l^{(k_t)}_{i_t,j_t}< 0$.
\end{lemma}

\begin{proof}
We prove the statement by induction on $\#\mathfrak{V}(\mathcal{C} )$. It is obvious if $\#\mathfrak{V}(\mathcal{C})=2$. Assume $\#\mathfrak{V}(\mathcal{C} )=n$. Let $(k,i,j)$ be maximal
and $\mathcal{C}'$ be the set obtained from $\mathcal{C}$ by RR-method i.e. removing all relations that involve $(k,i,j)$. By induction, there exist
 sequences $(k_t',i_t',j_t')$ $1\leq t \leq s$
such that for any $r\leq s$, $[l+\sum_{t=1}^{r}\epsilon_t\delta^{(k_t)}_{i_t,j_t}]$ satisfies  $\mathcal{C}'$
and $[l+\sum_{t=1}^{s}\epsilon_t\delta^{(k_t)}_{i_t,j_t}]=[\gamma+l^{(k)}_{ij}-\xi^{(k)}_{ij}]$.

If $l^{(k)}_{ij}-\xi^{(k)}_{ij}=m \geq 0$, set  $(k_t,i_t,j_t)=(k_t',i_t',j_t')$ for $1\leq t \leq s$, and $(k_t,i_t,j_t)=(k,i,j)$ for $s+1 \leq t\leq t+m$.

If $l^{(k)}_{ij}-\xi^{(k)}_{ij}=m < 0$, set $(k_t,i_t,j_t)=(k,i,j)$ for $ 1 \leq t\leq m$ and $(k_{m+t},i_{m+t},j_{m+t})=(k_t',i_t',j_t')$
for $1\leq t \leq s$.

The statement is proved.
\end{proof}

Now we can prove Theorem \ref{thm-A}.

\begin{theorem}\label{thm-irr}
 Let $\mathcal{C}$ be an admissible set of relations. The module $V_{\mathcal{C}}([l])$ is irreducible if and only if $\mathcal{C}$ is the maximal set of relations satisfied by $[l]$.
\end{theorem}

\begin{proof}
Suppose
$\mathcal{C}$ is not the maximal set of relations satisfied by $[l]$. Then there exist
$l_{r+1,i}^{(s)}-l_{r,j}^{(t)}\in \mathbb{Z}$ and there is not relation between $(s,r+1,i)$ and $(t,r,j)$.
So there exists tableau $[\gamma]\in W(\pi) [l]$ such that $\gamma_{r+1,i}^{(s)}-\gamma_{r,j}^{(t)}\in \mathbb{Z}_{\geq 0}$
and $\xi\in W(\pi) [l]$ such that $\xi_{r ,j}^{(t)}-\xi_{r+1,i}^{(s)}\in \mathbb{Z}_{>0}$.
By  Equation \eqref{action bc} one has that  $\xi$ is not in the submodule $W(\pi) [\gamma]$ of $V_{\mathcal{C}}([l])$ generated by $[\gamma]$, thus $V_{\mathcal{C}}([l])$ is not irreducible.


Conversely, let $\mathcal{C}$ be the maximal set of relations satisfied by $[l]$. By Lemma \ref{lemma sequences}, for any tableaux $[l]$ and $[\gamma]$, there exit $\{(k_t,i_t,j_t)\}$ $1\leq t \leq s$
such that for any $r\leq s$, $[l+\sum_{t=1}^{r}\epsilon_t\delta^{(k_t)}_{i_t,j_t}]$ satisfies  $\mathcal{C}$
and $[l+\sum_{t=1}^{s}\epsilon_t\delta^{(k_t)}_{i_t,j_t}]=[\gamma]$.
If $[l]$ and $[l+\delta^{(k)}_{ij}]$ satisfy $\mathcal{C}$,
then $l^{(k)}_{i,j}\neq \delta^{(t)}_{i+1,j'}$
for any $t,j'$. Similarly if $[l]$ and $[l-\delta^{(k)}_{ij}]$ satisfy $\mathcal{C}$, then $l^{(k)}_{i,j}\neq \delta^{(t)}_{i-1,j'}$ for any $t,j'$. Thus the coefficient of $[l+  \delta^{(k )}_{i ,j }]$ in
$e_{i }^{(p_{i+1}-p_i+1)}\ts [l]$ (resp. $[l-\delta^{(k )}_{i ,j }]$ in
$f_{i }^{(1)}\ts [l]$
is nonzero. By Lemma \ref{lemma separation},
$[l \pm\delta^{(k_1)}_{i_1,j_1}] \in V_{\mathcal{C}}([l])$. By the induction on $s$,
we conclude that $[\gamma] \in V_{\mathcal{C}}([l])$.
\end{proof}

Note that Theorem \ref{thm-irr} is a generalization of Proposition $5.3$    in \cite{FRZ}.

\subsection{Highest weight relation modules}

Denote by $q_{k}$ the number of bricks in the column $k$ of the pyramid $\pi$, $k=1, \ldots, l:=p_{n}$. We have $q_{1}\geq\cdots\geq q_{l}>0$, where $l=p_{n}$ is the number of the columns in $\pi$. Note that $N=q_1+q_2+\cdots +q_l=p_1 + \dots + p_n$, moreover, if $p_{i-1}<k\leq p_{i}$ for some $i\in\{1,\ \ldots,\ n\}$ (taking $p_{0}=0$), then $q_{k}=n-i+1$.
Let $\mathfrak{g}=\gl_N$, $\mathfrak{p}$ be the standard  parabolic subalgebra of $\mathfrak{g}$
with the Levi factor $\mathfrak{a}=\gl_{q_{1}}\oplus\cdots\oplus \gl_{q_{l}}$.
Then $W(\pi)$ is a subalgebra of $U(\mathfrak{p})$.
We will identify $U(\mathfrak{a})$ with $U(\gl_{q_{1}})\otimes\cdots\otimes U(\gl_{q_{l}})$. Let $\xi: U(\mathfrak{p})\rightarrow U(\mathfrak{a})$ be the algebra homomorphism induced by the natural projection $\mathfrak{p}\rightarrow \mathfrak{a}$. The restriction
\begin{equation*}
\bar{\xi}: W(\pi)\rightarrow U(\mathfrak{a})
\end{equation*}
of $\xi$ to $W(\pi)$ is called the Miura transform. By \cite{BK1}, Theorem 11.4, $\bar{\xi}$ is an injective algebra homomorphism, allowing us to view $W(\pi)$ as a subalgebra of $U(\mathfrak{a})$.

Let $M_k$ be a module for the Lie algebra $\gl_{q_{k}}$, $k=1, \ldots, l$. Then using  the Miura transform $\bar{\xi}$
the vector space
\begin{equation*}
M_1\otimes\ldots\otimes M_l
\end{equation*}
can be equipped with a module structure over the algebra $W(\pi)$.

For each $i=1, \ldots, n$, let $\mathcal{C}_i$ be an admissible set of relations for $\gl_{q_i}$  and $[L^{(i)}]$ be a tableau such that $\mathcal{C}_i$ is the maximal set of relations satisfied by $[L^{(i)}]$. Then $V_{\mathcal{C}_1}([L^{(1)}])\otimes \ldots \otimes V_{\mathcal{C}_l}([L^{(l)}])$ is a $\gl_{q_{1}}\oplus\ldots\oplus \gl_{q_{l}}$-module and thus a
$W(\pi)$-module.

In the following  we describe  a family of highest weight modules which can be realized as relation modules $V_{\mathcal{C}}([l])$ for some admissible sets of relations $\mathcal{C}$.

Let  $\lambda(u)=(\lambda_1(u), \ldots, \lambda_n(u))$, where $\lambda_i(u)=\prod\limits_{s=1}^{l} (u+\lambda_{s}^{(i)})$, $i=1, \ldots, n$.
We identify $\lambda_i(u)$ with the tuple $(\lambda_{1}^{(i)}, \ldots, \lambda_{l}^{(i)})$.

Denote  $[L]_{\lambda}=([l^{(1)}], \ldots, [l^{(l)}])$, where each
 $[l^{(k)}]$ is the tableau such that
$l_{ij}^{(k)}=l_{nj}^{(k)}=\lambda_j^{(k)} - j+1$, $i=1, \ldots, n$, $j=1, \ldots, i$, $k=1, \ldots, l$.

\begin{definition}
We will say that $\lambda_k(u)$ is \emph{good} if it satisfies the conditions: $\lambda_{i}^{(k)}-\lambda_j^{(k)}\notin \mathbb{Z}$ or $\lambda_{i}^{(k)}-\lambda_j^{(k)}> i-j$ for any $1\leq i<j\leq n-1$.
We say that $\lambda(u)$ is good if $\lambda_k(u)$ is good for all $k=1, \ldots, n$.
In this case  $[L]_{\lambda}$ is also called good.
\end{definition}

Assume $\lambda(u)$ is good.
For each $k=1, \ldots, l$ let
$\mathcal C_k$ be the maximal set of relations satisfied by
$[l^{(k)}]$, then $V_{C_k}([l^{(k)}])$ is the irreducible highest weight
 $\gl_{q_k}$-module with highest weight $\lambda^{(k)}=(\lambda_{n-q_k+1}^{(k)}, \cdots, \lambda_{n}^{(k)})$
 (\cite{FRZ} Proposition 5.7).

The following proposition follows from Theorem \ref{sufficiency of admissible} and Theorem \ref{thm-irr}.

\begin{proposition}\label{hw modules} Let $\lambda(u)$ be good, $[L]_{\lambda}=([l^{(1)}], \ldots, [l^{(t)}]$. Set $\mathcal C=\mathcal{C}_1\cup \ldots \cup \mathcal{C}_t$, where $\mathcal{C}_k$ is defined as above, $k=1, \ldots, t$.
If  for all  $i, j$, $r\neq s$ and $[T^{(r)}]\in V_{\mathcal{C}_r}([l^{(r)}])$, $[T^{(s)}]\in V_{\mathcal{C}_s}([l^{(s)}])$ we have $T^{(r)}_{ki}\neq T^{(s)}_{kj}$ for $k=1, \ldots, n$, then
$\mathcal{C}$ is admissible and $L(\lambda(u))\simeq V_{\mathcal{C}}([l])$.  Moreover, the explicit basis is $\mathcal{B}_{\mathcal{C}}([l])$.
\end{proposition}

In particular, if $\lambda(u)$ is a good dominant integral highest weight, then $L(\lambda(u))$ is a finite dimensional relation module. But we should note that not every finite dimensional $W(\pi)$-module is a relation module. For instance, if $t=2$,  $\lambda^{(1)}=\lambda^{(2)}=(5,1)$, then we have some equal entries in the first row. Hence, the corresponding finite dimensional module it is not a relation module.

\section{Tensor product of highest weight relation  modules}

If the tableau $\pi$ has parameters $p_1=\ldots = p_n=p$ then $W(\pi)$ is a finitely generated Yangian of level $p$. In this section we consider certain highest weight relation  modules
for the Yangians. It will be more convenient to work with the full Yangian.

The {\it Yangian} $\Y(n)=\Y(\gl_n)$, is the complex associative algebra with the generators $t_{ij}^{(1)}, t_{ij}^{(2)}, \ldots$ where $1\leq i, j\leq n$, and the defining relations
\begin{align}\label{Yre}
[t_{ij}(u),\displaystyle \ t_{kl}(v)]=\frac{1}{u-v}(t_{kj}(u)t_{il}(v)-t_{kj}(v)t_{il}(u))
\end{align}
where
$$
t_{ij}(u)=\delta_{ij}+t_{ij}^{(1)}u^{-1}+t_{ij}^{(2)}u^{-2}+\cdots\in \mathrm{Y}(n)[[u^{-1}]]
$$
and $u$ is a formal  variable. The Yangian $\Y(n)$ is a Hopf algebra with the coproduct
$\Delta : \Y(n)\rightarrow \Y(n)\otimes \Y(n)$ defined by
\begin{align}\label{Ycp}
\Delta(t_{ij}(u))=\sum_{a=1}^{n}t_{ia}(u)\otimes t_{aj}(u)
\end{align}
Given sequences $a_{1}, \ldots, a_{r}$ and $b_{1}, \ldots, b_{r}$ of elements of $\{$1, $\ldots, n\}$ the corresponding {\it quantum minor} of the matrix $[t_{ij}(u)]$ is defined by the following equivalent
formulas:
\begin{align*}
t_{b_{1}\cdots b_{r}}^{a_{1} \cdots a_{r}}(u)= \sum_{\sigma\in \mathfrak{S}_{r}} \text{ sgn } \sigma\cdot t_{a_{\sigma(1)}b_{1}}(u)\cdots t_{a_{\sigma(r)}b_{r}}(u-r+1)\\
=\sum_{\sigma\in \mathfrak{S}_{r}} \text{ sgn }  \sigma\cdot t_{a_{1}b_{\sigma(1)}}(u-r+1)\cdots t_{a_{r}b_{\sigma(r)}}(u).
\end{align*}
The series $t_{b_{1}\cdots b_{r}}^{a\cdots a_{r}}(u)$ is skew symmetric under permutations of the indices $a_{i}$, or $b_{i}$.

\begin{proposition}[\cite{NT} Proposition 1.11]
 The images of the quantum minors under the coproduct are given by
\begin{equation}\label{cpdet}
\Delta(t_{b_{1}\cdots b_{r}}^{a_1\cdots a_{r}}(u))=
\sum_{c_{1}<\cdots<c_{r}}
t_{c_{1}\cdots c_{r}}^{a_{1}\cdots a_{r}}(u)
\otimes t_{b_{1}\cdots b_{r}}^{c_{1}\cdots c_{r}}(u)\ ,
\end{equation}
summed over all subsets of indices $\{c_{1}, \ldots,\ c_{r}\}$  from $\{1, \ldots, n\}.$
\end{proposition}
For $m\geq 1$ introduce the series $a_{m}(u), b_m(u)$ and $c_m(u)$ by
\begin{align*}
a_{m}(u)=t_{1\cdots m}^{1\cdots m}(u),\ b_{m}(u)=t_{1\cdots m-1,m+1}^{1\cdots m}(u),\
c_{m}(u)=t^{1\cdots m-1,m+1}_{1\cdots m}(u)
\end{align*}
The coefficients of these series generate the algebra $\Y(n)$, they are called the {\it Drinfeld generators}.
\begin{definition}
Let $V$ be a $Y(n)$-module. A nonzero $v\in V$ is called singular if:

\begin{itemize}
\item[(i)] $v$ is a weight vector.
\item[(ii)] $b_{m}(u)v=0$ for any $m\geq 1$.
\end{itemize}
\end{definition}

Let $E_{ij}, i, j=1, \ldots, n$ denote the standard basis elements of the Lie algebra $\gl_n$. We have a natural
 embedding
\begin{equation*}
U(\gl_{n})\rightarrow \mathrm{Y}(n)\ ,\ E_{ij}\mapsto t_{ij}^{(1)}.
\end{equation*}

Moreover, for any $a\in \mathbb{C}$ the mapping
\begin{align}\label{ev}
\varphi_{a}:t_{ij}(u)\mapsto\delta_{ij}+\frac{E_{ij}}{u-a}
\end{align}
defines an algebra epimorphism from $\mathrm{Y}(n)$ to the universal enveloping algebra $\mathrm{U}(\gl_n)$ so that any $\gl_n$-module can be extended to a $\mathrm{Y}(n)$-module via \eqref{ev}.
 Consider the irreducible $\gl_{n}$-module $L(\lambda)$ with  highest weight $\lambda=(\lambda_{1},\ \ldots,\ \lambda_{n})$ with respect to the upper triangular Borel subalgebra generated by $E_{ij}$, $i<j$. The corresponding $\mathrm{Y}(n)$-module is denoted by $L_{a}(\lambda)$, and we call it the {\it evaluation module}. We keep the notation $L(\lambda)$ for the module $L_{a}(\lambda)$ with $a=0$. The coproduct $\Delta$ defined by \eqref{Ycp} allows one to consider the tensor products
$
L_{a_1}(\lambda^{(1)})\otimes L_{a_2}(\lambda^{(2)})\otimes\cdots\otimes L_{a_l}(\lambda^{(l)})
$
as $\mathrm{Y}(n)$-modules.

Let
 $L$ be a $\gl_n$-module   with finite dimensional
weight subspaces,

\begin{equation*}
L=\bigoplus_{\mu}L_{\mu}, \text{ dim} L_{\mu}<\infty.
\end{equation*}

Then we define the restricted dual  to $L$ by
\begin{equation*}
L^{*}=\bigoplus_{\mu}L_{\mu}^*.
\end{equation*}

The elements of $L^{*}$ are finite linear combinations of the vectors dual to the basis vectors of any weight basis of $L$. The space $L^{*}$ can be equipped with a $\gl_n$-module structure by
\begin{equation*}
(E_{ij}f)(v)=f(-E_{n-i+1,n-j+1}v)\ ,\ f\in L^{*},\ v\in L.
\end{equation*}

Denote by $\omega$ the anti-automorphism of the algebra $\mathrm{Y}(n)$ , defined by
\begin{equation*}
\omega :t_{ij}(u)\mapsto t_{n-i+1,n-j+1}(-u).
\end{equation*}

Suppose now that the $\gl_{n}$ action on $L$ is obtained by the restriction of an action of $\Y(n)$. Then the $\gl_{n}$-module structure on $L^{*}$ can be regarded as the restriction of the $\Y(n)$-module structure defined by
\begin{equation*}
(xf)(v)=f(\omega (x) v)\ ,\text{ for } x\in\Y(\gl_n) \text{ and } \ f\in L^{*},\ v\in L.
\end{equation*}

For any $\lambda=(\lambda_{1},\ \ldots\ ,\ \lambda_{n})$ we set $\tilde{\lambda}=(-\lambda_{n},\ \ldots,\ -\lambda_{1})$. Then we have

 \begin{proposition}\cite{m:yc}\label{dual module}
 Let $L$ be the tensor product
 $L(\lambda^{(1)})\otimes L(\lambda^{(2)})\otimes\cdots\otimes L(\lambda^{(l)})$.
 Then the  $\Y(\gl_n)$-module $L^{*}$  is isomorphic to the tensor product module
\begin{equation*}
L(\tilde{\lambda}^{(1)})\otimes L(\tilde{\lambda}^{(2)})\otimes\ldots\otimes L(\tilde{\lambda}^{(l)})\ .
\end{equation*}
\end{proposition}

 \begin{proposition}\label{piso}
 Suppose that the $\Y(n)$-module
 \begin{equation}\label{eq-piso}
L(\lambda^{(1)})\otimes L(\lambda^{(2)})\otimes\cdots\otimes L(\lambda^{(l)})
\end{equation}
 is irreducible.
 Then any permutation of the tensor factors  gives an isomorphic representation of $\Y(n)$.
\end{proposition}
\begin{proof}
 Denote the tensor product  by $L$. Note that $L$ is a $Y(n)$-module representation with highest weight $(\lambda_{1}(u),\ \ldots,\ \lambda_{n}(u))$. Consider a representation $L'$ obtained by a certain permutation of the tensor factors in  \eqref{eq-piso}. The tensor product $\zeta'$ of the highest vectors of the representations $L(\lambda^{(i)})$ is a singular vector in $L'$ whose weight is same as the highest weight of $L$. This implies that $\zeta'$ generates a highest weight submodule in $L'$ such that its irreducible quotient is isomorphic to $L$. However, $L$ and $L'$ have the same formal character as $\gl_n$-modules  which implies that $L$ and $L'$ are isomorphic.
\end{proof}

\section{Irreducibility of tensor product}

In this section we discuss the irreducibility of tensor product of  relation highest weight modules for the Yangians.
We consider the  tensor product of  $\gl_n$-highest weight modules $L(\lambda)$'s with good $\lambda$'s.

Let $\lambda^{(i)} =(\lambda_{1}^{(i)},\ \ldots,\ \lambda_{n}^{(i)})$, $i=1, \ldots, l$ be $n$-tuples of complex numbers. We will call the set $\{\lambda^{(1)},\ \ldots,\ \lambda^{(l)}\}$ \emph{generic} if  for each pair of indices
$1\leq i<j\leq l$  we have
$\lambda_s^{(i)} - \lambda_t^{(j)}\notin \mathbb Z$, $s, t=1, \ldots, n$.

Denote by $L(\lambda^{(i)})$ the simple  $\gl_{n}$-module with highest weight $\lambda^{(i)}$, $i=1, \ldots, l$.  Also, by $\mathcal{B}(\lambda^{(i)})$ we will denote the basis of tableaux of $L(\lambda^{(i)})$ guarantied by Gelfand-Tsetlin Theorem (see \cite{GT}).
Our first result is the irreducibility of tensor product $L(\lambda^{(1)})\otimes L(\lambda^{(2)})\otimes\cdots\otimes L(\lambda^{(l)})$ in the generic case.

\begin{theorem}\label{thm-generic}
Let $\{\lambda^{(1)},\ \ldots,\ \lambda^{(l)}\}$ be a generic set with good $\lambda^{(i)}$, $i=1, \ldots, l$.
Then  the $\Y(n)$-module
$$
L(\lambda^{(1)})\otimes L(\lambda^{(2)})\otimes\cdots\otimes L(\lambda^{(l)})
$$
 is irreducible.
\end{theorem}

We will say that $\lambda(u)$
is \emph{integral} if it is not generic.

We will establish the sufficient conditions of irreducibility of the $\Y(\gl_n)$-module $L(\lambda)\otimes L(\mu)$ with good integral $\lambda$ and $\mu$.
This extends the result of \cite{M1} to  some infinite dimensional highest weight modules, though unlike in \cite{M1} we can not show the necessity of these conditions for the irreducibility of the tensor product, neither can we prove it for any number of tensor factors.

For any pair of indices $i < j$  the set $\{l_j,l_{j-1},\ldots,l_i\}$ is the union of
pairwise disjoint sets
$\{l_{i_{11}}, \ldots, l_{i_{1m_1}}\}, \{l_{i_{21}}, \ldots, l_{i_{2m_2}}\},\ldots,\{l_{i_{t1}}, \ldots, l_{i_{tm_t}}\}$
such that $l_{i_{ra}}-l_{i_{sb}}\notin \mathbb Z$ for any $r\neq s$, $l_{i_{ra}}-l_{i_{rb}}\in \mathbb Z$ and
$i_{ra}>i_{rb}$ for $a<b$.

We shall denote
\begin{align*}
\lfloor l_{i_{r1}},l_{i_{rm_r}}\rfloor^-=\left\{
\begin{array}{cc}
\{l_{i_{r1}},l_{i_{r1}}+1,\ldots,l_{i_{rm_r}}\}\setminus\{l_{i_{r1}},l_{i_{r2}},\ldots,l_{i_{rm_r}}\}, i_{rm_1}=j\\
\{l_{i_{rm_1}}+z\ |\ z\in \mathbb{Z}_{\leq 0}\} \setminus\{l_{i_{r1}},l_{i_{r2}},\ldots,l_{i_{rm_r}}\},
 i_{rm_1}\neq j
\end{array}
\right.
\end{align*}
and
\begin{align*}
\lfloor l_{i_{r1}},l_{i_{rm_r}}\rfloor^+=\left\{
\begin{array}{cc}
\{l_{i_{r1}},l_{i_{r1}}+1,\ldots,l_{i_{rm_r}}\}\setminus\{l_{i_{r1}},l_{i_{r2}},\ldots,l_{i_{rm_r}}\}, i_{rm_1}=i\\
\{l_{i_{rm_r}}+z\ |\ z\in \mathbb{Z}_{\geq 0}\} \setminus\{l_{i_{r1}},l_{i_{r2}},\ldots,l_{i_{rm_r}}\},
 i_{rm_r}\neq i
\end{array}
\right.
\end{align*}

\begin{align*}
\begin{array}{cc}
\langle l_j,l_i\rangle ^-=\cup_{r=1}^{t} \lfloor l_{i_{r1}},l_{i_{rm_r}}\rfloor^-\\
\langle l_j,l_i\rangle^+=\cup_{r=1}^{t} \lfloor l_{i_{r1}},l_{i_{rm_r}}\rfloor^+
\end{array}
\end{align*}

Let $\lambda=(\lambda_1, \ldots, \lambda_n)$  and $\mu=(\mu_1, \ldots, \mu_n)$ are $n$-tuples of complex numbers. Consider irreducible highest weight $\gl_{n}$-modules $L(\lambda)$ and $L(\mu)$ with highest weights $\lambda$ and
$\mu$ respectively.

Set
$$
l_i=\lambda_i- i+1,\  m_i = \mu_i- i+1,\  i=1,\cdots, n.
$$

\begin{theorem}\label{thm-integral}
Let $\lambda$ and $\mu$ be good integral $\gl_n$-highest weights.
 Suppose that for each pair of indices
$1\leq i < j\leq n$  we have
\begin{align}\label{condition}
m_{j}\notin\langle l_{j},  l_{i}\rangle^-,
m_{i}\notin\langle l_{j},  l_{i}\rangle^+
 \text{ or } l_{j} \notin\langle m_{j},  m_{i}\rangle^-, l_{i} \notin\langle m_{j},  m_{i}\rangle^+.
\end{align}
Then the $\Y(\gl_n)$-module $L(\lambda)\otimes L(\mu)$  is irreducible.

\end{theorem}

In the following we prove Theorem \ref{thm-generic} and Theorem \ref{thm-integral}. The proofs closely follow the  proof of Theorem 3.1 in \cite{M1}  for finite dimensional modules.
We include the details for completeness.

\subsection{Integral case}

 We start with the proof of Theorem \ref{thm-integral}. Assume that $L(\lambda)\otimes L(\mu)$ is not irreducible as  $\mathrm{Y}(n)$-module.
Let $\xi$ and $\xi'$ denote the highest weight vectors of the $\gl_{n}$-modules $L(\lambda)$ and $L(\mu)$, respectively.
Consider a nonzero $\mathrm{Y}(n)$-submodule $N$ of $L(\lambda)\otimes L(\mu)$. Then  $N$ must contain a nonzero singular  vector $\zeta$.
We will show by induction on $n$ that
$
\zeta\in \mathbb{C}\cdot \xi\otimes\xi'.$ Since $\lambda$ is  good then $L(\lambda)$ is a relation  $\gl_{n}$-module by \cite{FRZ}, Proposition 5.7.
We denote by $H$ the Cartan subalgebra  of $\gl_n$ consisting of diagonal matrices. We identify an element  $w\in \mathfrak{h}^*$ with the $n$-tuple consisting of values of $w$ on the standard
basis of $\mathfrak{h}$.

Consider the Gelfand-Tsetlin basis $\mathcal{B}(\lambda)$ of  $L(\lambda)$.  The tableau corresponding to the element $\xi$ is of the form $[L]=(l_{ij})$
with $l_{ij}=\lambda_j-j+1$ for all $i=1, \ldots, n$, $j=1, \ldots, i$.


The element
 $\zeta$ can be  written uniquely as a finite sum:
\begin{align}\label{singular vector}
\zeta=\sum_{[L]\in\mathcal{B}(\lambda)}[L]\otimes m_{L},
\end{align}
where $m_{L}\in L(\mu)$.

Viewing $L(\lambda)\otimes L(\mu)$ as a $\gl_{n}$-module we immediately see that $\zeta$ is  a weight $\gl_{n}$-singular vector, that is $E_{ij}\zeta=0$ for all $i<j$.
 Moreover, all elements $[L]\otimes m_{L}$ in \eqref{singular vector} have the same $\gl_{n}$-weight.

If $[L]=(l_{ij})$ then the weight $w(L)$ of $[L]$ is a sequence $$\left\{\sum_{i=1}^{k}l_{ki}-\sum_{i=1}^{k-1}l_{k-1,i}+k-1, \,\ \  k=1, \ldots, n\right\}.$$
 Given two weights $w, w'\in \mathfrak{h}^{*}$, we shall write $w\preceq w'$ if $w'-w$ is a $\mathrm{Z}_{\geq 0}$-linear combination of the simple roots of
$\gl_{n}$. This defines a partial order on the set of weights of $\gl_{n}$.

Denote by $supp \, \zeta$ the set of tableaux $[L]\in\mathcal{B}(\lambda)$ for which $m_L\neq 0$ in (\ref{singular vector}).
Let $[L^{0}]$ be a minimal element  in  $supp \, \zeta$  with respect to the partial ordering on the weights $w(L)$'s .

Since
$t_{1\ldots m-1, m+1}^{1\cdots m}(u) \zeta=0,$ we have

\begin{align}\label{bmzeta}
\sum_{c_{1}<\cdots<c_{m}}\sum_{L}t_{c_{1}\ldots c_{m}}^{1\cdots m}(u) [L]
\otimes t_{1\ldots m-1, m+1}^{1\cdots m}(u) m_{L}=
\end{align}

$$=t_{1\ldots m}^{1\cdots m}(u)[L^0]\otimes t_{1\ldots m-1, m+1}^{1\cdots m}(u) m_{L^0} + \ldots =0. $$
Hence, $t_{1\ldots m-1, m+1}^{1\cdots m}(u) m_{L^0}=0$ for all $m$.  Thus $m_{L^0}$ is a highest vector of $L(\mu)$ and we conclude
that $m_{L^{0}}$ is a scalar multiple of $\xi'$. This immediately implies
 that  $L^{0}$  is determined uniquely.  For any
 $L\in supp \, \zeta$ we have $w(L)\succeq w(L^{0})$.  If $[L]=(l_{ij})\in supp \, \xi$ and  $[L^{0}]=(l^0_{ij})$ then
 we have $l_{ij} - l_{ij}^{0}\in \mathrm{Z}_{\geq 0}$ for $1\leq j\leq i\leq n-1.$


 Permuting $L(\lambda)$ and $L(\mu) $ if necessary and applying Proposition \ref{piso}, we  assume that $m_{n}\notin\langle l_{n},  l_{1}\rangle^-,
m_{1}\notin\langle l_{n},  l_{1}\rangle^+$.





\begin{lemma}\label{lemma4}
The $(n-1)$-th row of  $[L^{0}]$ is $(l^0_{1},\ \ldots,\ l^0_{n-1})$, where $l^0_i=\lambda_i-i+1$.
\end{lemma}

\begin{proof}
Suppose the contrary. Then for each $j$ with $l^{0}_{n-1, j}\neq l_{j}^{0}$ there exists a minimal $r(j)$ such that $[L'(r(j))]=[L^{0}+ \delta_{n-1, j_1}+\cdots+ \delta_{n-r(j), j_{r(j)}}]$ is a Gelfand-Tsetlin tableau of $L(\lambda)$ with $j_1=j$. Choose $j$ such that $r(j)$ is minimal and denote it by $r$. Also set $L'=L'(r)$.

Since $\zeta$ is a singular vector, we have
\begin{align*}
t_{1,\cdots,n-r-1,n}^{1,\cdots,n-r}(u)\zeta=0
\end{align*}
and hence, by \eqref{cpdet} we have

\begin{align}\label{bmzet}
\sum_{c_{1}<\cdots<c_{n-r}}\sum_{L}t_{c_{1}\ldots c_{n-r}}^{1\cdots n-r}(u) [L]
\otimes t_{1\cdots n-r-1,n}^{c_{1}\cdots c_{n-r}}(u) m_{L}=0.
\end{align}

Following the proof of Lemma 3.5 in \cite{M1} we look at the coefficient of
 $[L']\otimes m_{L^{0}}$ in the expansion of the left hand side. It comes from the following
 two summands in
\eqref{bmzet}:

\begin{align}\label{e1}
t_{1\ldots n-r-1,n}^{1\cdots n-r}(u)[L^{0}]
\otimes t_{1\cdots n-r-1,n}^{1\ldots n-r-1,n}(u)m_{L^0}
\end{align}
and
\begin{align}\label{e2}
t_{1\ldots n-r}^{1\cdots n-r}(u)[L']
\otimes t_{1\cdots n-r-1,n}^{1\ldots n-r}(u)m_{L'},
\end{align}

 if $[L']\in supp \, \zeta$.



Consider \eqref{e1} first.
Due to  the minimality of $r$ we have $E_{in}[L^{0}]=0$ for $n-r<i \leq n-1$. Hence
\begin{align*}
E_{n-r,n}[L^{0}]=(-1)^{r-1}E_{n-1,n}E_{n-2,n-1}\cdots E_{n-r,n-r+1}[L^{0}].
\end{align*}
Therefore the expansion of $E_{n-r,n}[L^{0}]$  contains a term $a[L']$ with $a\neq 0$.


It will  be convenient to use polynomial quantum minors  defined by:
\begin{align*}
T_{i_{1}\cdots i_{m}}^{j_1\cdots  j_{m}}(u)=u(u-1)\cdots(u-m+1)
t_{i_{1}\cdots i_{m}}^{j_1\cdots  j_{m}}(u).
\end{align*}
Then
the coefficient of $[L']$ in
$T_{1\ldots n-r-1,n}^{1\cdots n-r}(u)[L^{0}]$ equals
\begin{align*}
a(u+l_{n-r,1}^{0})\cdots\bigwedge_{i_r}\cdots(u+l_{n-r,n-r}^{0})\ .
\end{align*}

On the other hand,

\begin{align*}
T_{1\cdots n-r-1,n}^{1\ldots n-r-1,n}(u)m_{L^0}
=(u+m_{1})\cdots(u+m_{n-r-1})(u+m_{n}+r)m_{L^0}.
\end{align*}

Hence,

\begin{align}\label{e11}
T_{1\ldots n-r-1,n}^{1\cdots n-r}(u)[L^{0}]
\otimes T_{1\cdots n-r-1,n}^{1\ldots n-r-1,n}(u)m_{L^0}=
\end{align}

$$a(u+l_{n-r,1}^{0})\cdots\bigwedge_{i_r}\cdots(u+l_{n-r,n-r}^{0}) (u+m_{1})\cdots(u+m_{n-r-1})(u+m_{n}+r)([L']\otimes m_{L^0}).$$

Consider now \eqref{e2}. We have
\begin{align*}
T_{1\cdots n-r}^{1\cdots n-r}(u)[L']
=(u+l_{n-r,1}^{0})\cdots(u+l_{n-r,i_r}^{0}+1)\cdots(u+l_{n-r,n-r}^{0})[L'].
\end{align*}


Let $[L]_{\mu}$ be the highest weight tableau of $L(\mu)$ in the Gelfand-Tsetlin realization of $L(\mu)$. Then $m_{L^0}$ is a multiple of $[L]_{\mu}$.
Comparing the weights of $[L^0]\otimes m_{L^0}$ and $[L']\otimes m_{L'}$
we see that
$m_{L'}$ is a multiple of the tableau
 $[L]_{\mu, r}= [L]_{\mu} - \delta_{n-1, j_1} - \cdots  - \delta_{n-r, j_r} $. Since $(n-r, j)\geq(n-r-1, j)$ for $j=1,\ldots, n-r-1$ and the $(n-r-1)$-th row of each patter is $(\mu_1,\ .\ .\ .\ ,\ \mu_{n-r-1})$, we have that
$j_r=n-r$ and the $(n-r)$-th row of $[L]_{\mu, r}$ is $(m_1,\ .\ .\ .\ ,\ m_{n-r-1},\ m_{n-r}-1)$.

Therefore $E_{n-r, n}m_{L'}$ is a scalar multiple of $m_{L^0}$.
If $[L']$ is not in $supp \, \zeta$ then $m_{L'}=0$.
In both cases we have that $E_{n-r, n}m_{L'}= b m_{L^0}$ for some constant $b$, and so

\begin{align*}
T_{1\ldots n-r-1,n}^{1\cdots n-r}(u) m_{L'}
=b \cdot (u+m_{1})\cdots(u+m_{n-r-1}) m_{L^0}.
\end{align*}

We have

\begin{align*}
T_{1\cdots n-r}^{1\cdots n-r}(u)[L']\otimes T_{1\ldots n-r-1,n}^{1\cdots n-r}(u) m_{L'}
=
\end{align*}

$$b \cdot (u+m_{1})\cdots(u+m_{n-r-1})(u+l_{n-r,1}^{0})\cdots(u+l_{n-r,i_r}^{0}+1)\cdots(u+l_{n-r,n-r}^{0})([L']\otimes m_{L'}).$$

Combining these results   we obtain

$$
a(u+m_{n}+r)+ b \cdot (u+l_{n-r,i_r}^{0}+1)=0.
$$

In particular, we have $b=-a\neq 0$ and
 $m_{n}=l_{n-r,i_r}^{0}-r+1$. By the minimality of $r$ we have $l_{n-s,i_s}^0=l_{n-s+1,i_{s-1}}^0+1$ and $m_{n}=l_{n-1,i}^{0}$.

By the definition of $[L^{0}]$      we have
$l_{i}-l_{n-1,i}^{0}\geq 0$ and $l_{n-1,i}^{0}-l_{k}\geq 0$, where
$k$ is the minimal index such that $k>i$ and
$\lambda_i-\lambda_{k}\in \mathbb Z_{\geq 0}$. This implies
 $l_{i}-m_{n}\in Z_{> 0}$ and $m_{n}-l_{k}\in Z_{> 0}$.
Thus  $m_n\in \langle l_{n},  l_{1}\rangle^-$, which is a contradiction.
This completes the proof of the lemma.
\end{proof}

Lemma \ref{lemma4}  implies that all tableaux $[L]\in supp \, \zeta$ belong to the $\gl_{n-1}$-submodule $L(\lambda_{-})$ of $L(\lambda)$ generated by  $\xi$.
Note that he module $L(\lambda_{-})$ is irreducible  with the highest weight $\lambda_{-}=(\lambda_{1},\ \ldots,\ \lambda_{n-1})$ by \cite{FRZ}, Proposition 5.3.
We have
 $E_{nn}[L]=\lambda_{n}[L]$ for all $[L]\in supp \, \zeta$. Moreover, $$
w(L)+w(m_L)=w(L^{0})+\mu.
$$
Hence, $E_{nn}m_L=\mu_{n}m_L$ and   the $(n-1)$-th row of each tableau $m_L$ coincides with $(\mu_{1},\ \ldots,\ \mu_{n-1})$. We see that each $m_L$  belongs to the $\gl_{n-1}$-submodule $L(\mu-)$ generated by $\xi'$, which is irreducible highest weight module with the highest weight $\mu_{-}= (\mu_1,\ldots,\ \mu_{n-1})$. Therefore, $\zeta\in L(\lambda_{-})\otimes L(\mu_{-})$.

The $\mathrm{Y}(n-1)$-module structure on
$L(\lambda_{-})\otimes L(\mu_{-})$ coincides the with the one obtained by restriction from $\mathrm{Y}(n)$ to the subalgebra generated by the $t_{ij}(u)$ with $1\leq i, j\leq n-1$ by \eqref{Ycp} and  \eqref{Yre}.
The vector $\zeta$ is singular for $\mathrm{Y}(n-1)$ (it is annihilated by  $b_{1}(u), \ldots, b_{n-2}(u)$). By the assumption of the theorem, for each pair $(i,\ j)$ such that $1\leq i<j\leq n-1$ the condition \eqref{condition} is satisfied. Therefore $L(\lambda_{-})\otimes L(\mu_{-})$ is irreducible $\mathrm{Y}(n-1)$-module by the induction hypothesis.
Hence, $\zeta$ is a scalar multiple of $\xi\otimes\xi'$.

It remains to show that  $L(\lambda)\otimes L(\mu)$ is generated by  $\xi\otimes\xi'$.
Suppose that  $\xi\otimes\xi'$ generates a proper submodule $N$ in $\mathcal L=L(\lambda)\otimes L(\mu)$.  Denote
\begin{align*}
\tilde{N}=\{f\in L^{*}\ |\ f(v)=0 \text{ for all } v\in N\}.
\end{align*}
Then $\tilde{N}$ is a nonzero (since $N\neq\mathcal L$) submodule of $\mathcal L^{*}$. By Proposition \ref{dual module}
and above argument, $\tilde{N}$
 contains a singular vector $\zeta.$  As it was shown above $\zeta$ is a scalar multiple of  $\xi^*\otimes \xi'^{*}$ of the highest weight vectors of  $L(\lambda)^*$ and $ L(\mu)^*$
 respectively. On the other hand, $\xi^*\otimes \xi'^{*}\notin \tilde{N}$ giving a contradiction.  Hence, $\xi\otimes\xi'$ generates $L(\lambda)\otimes L(\mu)$. Since all singular elements of the highest weight $\Y(n)$-module
 $L(\lambda)\otimes L(\mu)$ belong to $\mathbb C \cdot \xi\otimes\xi'$, the module $L(\lambda)\otimes L(\mu)$ is irreducible.
This completes the proof of Theorem \ref{thm-integral}.

\subsection{Generic highest weight modules}

Now we prove  Theorem \ref{thm-generic} by induction on $l$.  The case $l=2$ is a consequence of Theorem \ref{thm-integral}, since the conditions of Theorem \ref{thm-integral} trivially
follow from the conditions of Theorem \ref{thm-generic}.

We assume now that $l > 2$ and denote by $K$  the tensor product $L(\lambda^{(2)})\otimes\cdots\otimes L(\lambda^{(l)})$.
Suppose that $K$ is irreducible highest weight  $\Y(n)$-module. We will show that $\mathcal L=L(\lambda^{(1)})\otimes K$ is irreducible. Then Theorem \ref{thm-generic} follows by induction.

The proof of irreducibility of $\mathcal L$  is similar to the proof of Theorem \ref{thm-integral}.
Suppose   $N$ is a nonzero $\mathrm{Y}(n)$-submodule of $\mathcal L$.
Then $N$ must contain a singular vector $\zeta$:

\begin{align} \label{generic singular}
\zeta=\sum_L [L]\otimes m_L,
\end{align}
summed over finitely many Gelfand-Tsetlin tableaux $[L]$ of $L(\lambda^{(1)})$, where $m_L\in K$.

Following the proof of Theorem \ref{thm-integral} we choose a minimal element $[L^{0}]$ of  the set of tableaux $[L]$ occurring in \eqref{generic singular}
 with respect to the partial ordering on the weights $w(\Lambda)$. As before $[L^{0}]$  is determined uniquely,  $m_{L^{0}}$ is a scalar multiple of $\xi'$ and for any $[L]$ that
 occurs in \eqref{generic singular}  $w(L)\succeq w(L^{0})$. Moreover,
 for each entry $l_{ij}$ of  $[L]$  occurring in \eqref{generic singular} we have $l_{ij}- l_{ij}^{0}\in \mathrm{Z}_{\geq 0}$, for $1\leq j\leq i\leq n-1.$

We also have an analog of Lemma \ref{lemma4}

\begin{lemma}\label{lemma4'}
The $(n-1)$-th row of  $[L^{0}]$ is $(l^0_{1},\ \ldots,\ l^0_{n-1})$, where $l^0_i=\lambda_i-i+1$.
\end{lemma}

\begin{proof}
Choose $[L']$ as in the proof of Lemma \ref{lemma4}.
Since $\zeta$ is a singular vector, we have
$$
0=T_{1,\cdots,n-r-1,n}^{1,\cdots,n-r}(u)\zeta=$$
$$
=\sum_{c_{1}<\cdots<c_{n-r}}\sum_{L}T_{c_{1}\ldots c_{n-r}}^{1\cdots n-r}(u)[L]
\otimes T_{1\cdots n-r-1,n}^{c_{1}\cdots c_{n-r}}(u)m_L.
$$

The coefficient of  $[L']\otimes m_{L^{0}}$ in the expansion of the left hand side of \eqref{singular vector} is the following

$$
(u+l_{n-r,1}^{0})\cdots\bigwedge_{i_r}\cdots(u+l_{n-r,n-r}^{0})
\prod_{i=2}^k (u+m_{1}^{(i)})\cdots(u+m_{n-r-1}^{(i)}) \times$$

$$ (a\prod_{i=2}^k (u+m_{n}^{(i)}+r) +g(u) (u+l_{n-r,i_r}^{0}+1))=0,$$
where $a\neq 0$ and $g(u)$ is a certain polynomial in $u$.




Put $u=-l_{n-r,i_r}^{0}-1 $. Since $a$ is nonzero, we get
$m_n^{(j)}=l_{n-r,i_r}^{0}-r+1=l_{n-1,i}^{0}$ for some $2\leq j \leq k$.
Thus $\lambda_n^{(j)}-\lambda_i^{(1)}\in \mathbb Z$  which is a contradiction. The lemma is proved.
\end{proof}

It remains to show that $\xi\otimes\xi'$  generates $\mathcal L$. The argument is the same as in the proof of Theorem \ref{thm-integral}.
This completes the proof of Theorem \ref{thm-generic}.

Let  $\lambda(u)=(\lambda_1(u), \ldots, \lambda_n(u))$, where $\lambda_i(u)=\prod\limits_{s=1}^l (u+\lambda_{s}^{(i)})$, $i=1, \ldots, n$. Set $\lambda^{(i)}=(\lambda_1^{(i)}, \ldots, \lambda_l^{(i)})$,
$i=1, \ldots, n$.

\begin{corollary}\label{thm-hw-irr} If $\lambda(u)$ is as in  Theorem \ref{thm-generic} then
$$L(\lambda(u))\simeq L(\lambda^{(1)})\otimes \cdots \otimes L(\lambda^{(n)}).$$
\end{corollary}

\begin{proof}
By  Theorem \ref{thm-generic} the tensor product  $V_{\mathcal{C}^{(1)}}([L_1])
\otimes\cdots\otimes V_{\mathcal{C}^{(l)}}([L_l]) $ is irreducible. Moreover, this  $\Y(n)$-module  contains a highest weight vector with weight $\lambda(u)$ which implies the statement.
\end{proof}

\begin{remark}
We can combine Theorem \ref{thm-generic}
 and Theorem \ref{thm-integral} and obtain irreducibility of the tensor product
 $$L(\lambda)\otimes L(\mu)\otimes L(\nu_1)\otimes \ldots \otimes L(\nu_s),$$
 where $\lambda$ and $\mu$ satisfy the conditions of Theorem \ref{thm-integral} and $\nu_1, \ldots, \nu_s$ satisfy the conditions of Theorem \ref{thm-generic} and
 $\nu_i^{j}-\lambda_k\notin \mathbb Z$,  $\nu_i^{j}-\mu_k\notin \mathbb Z$ for all possible $i, j, k$.

 \end{remark}

\section{Proof of Theorem \ref{sufficiency of admissible}}

Let $\mathcal{C}$ be a pre-admissible set of relations  and $[L]$  a tableau satisfying $\mathcal{C}$. Assume that $\mathcal{C}$ is a union of indecomposable sets from $\mathfrak{F}$.  We will show that for any defining relation $g=0$ in $W(\pi)$ and $[l]\in\mathcal{B}_{\mathcal{C}}([L])$ holds
$g[l]=0$.
Recall that the action of generators of $W(\pi)$ in $V_{\mathcal C}([L])$ is given by
\eqref{action of generators}, where
$
d_{r}^{(t)}[l]=d_r^{\ts(t)}(l)[l]$ and the action of $d_{r}^{\tss\prime\ts(t)}$ on $[l]$ is a multiplication by a scalar which is polynomial in $l$. Also recall that  the vector
 $[l\pm\de_{ri}^{(k)}]$ is zero if it does not satisfy $\mathcal{C}$.

Set

\begin{equation*}
e_{r,k,i}^{(t)}(l)=\left\{
\begin{array}{cc}
- \frac{\prod\limits_{j,t}(l_{r+1,j}^{(t)}-l_{r,i}^{(k)})}{\prod\limits_{(j,t)\neq (i,k)}(l_{r,j}^{(t)}-l_{r,i}^{(k)})}
b_{r,k,i}^{(t)}(L),&
\text{ if } [l]\in  \mathcal {B}_\mathcal{C}([L])
\\
0,& \text{ if }[l]\notin  \mathcal {B} _\mathcal{C} ([L]),
\end{array}
\right.
\end{equation*}

\begin{equation*}
f_{r,k,i}^{(t)}(l)=\left\{
\begin{array}{cc}
\frac{\prod\limits_{j,t}(l_{r-1,j}^{(t)}-l_{r,i}^{(k)})}{\prod\limits_{(j,t)\neq (i,k)}(l_{r,j}^{(t)}-l_{r,i}^{(k)})}
c_{r,k,i}^{(t)}(l),&
\text{ if } [l]\in  \mathcal {B}_\mathcal{C}([L])
\\
0,& \text{ if }[l]\notin  \mathcal {B} _\mathcal{C} ([L]),
\end{array}
\right.
\end{equation*}

\begin{equation*}
\Phi(l,z_1,\ldots,z_m)=
\left\{
\begin{array}{cc}
1,& \text{ if } [l+z_1+\ldots+z_t]\in  \mathcal{B}_{\mathcal{C}}([L]) \text{ for any } 1\leq t \leq m \\
0,& \text{ otherwise}.
\end{array}
\right.
\end{equation*}

Note that
$e_{r,k,i}^{(t)}(l)$ and $f_{r,k,i}^{(t)}(l)$ are rational functions in the components of $[l]$ and

$e_{r,k,i}^{(p_{r+1}-p_r+1)}(l)=- \frac{\prod\limits_{j,t}(l_{r+1,j}^{(t)}-l_{r,i}^{(k)})}{\prod\limits_{(j,t)\neq (i,k)}(l_{r,j}^{(t)}-l_{r,i}^{(k)})}$,
$f_{r,k,i}^{(1)}(l)=\frac{\prod\limits_{j,t}(l_{r-1,j}^{(t)}-l_{r,i}^{(k)})}{\prod\limits_{(j,t)\neq (i,k)}(l_{r,j}^{(t)}-l_{r,i}^{(k)})}$.

Now
the action of generators can be written as follows:
 \begin{align}\label{gt formula}
d_{r}^{(t)}[l]&=d_r^{\ts(t)}(l)[l],\\
e_r^{(t)}\ts [l] &=\sum_{i,k}\Phi(l,\de_{ri}^{(k)}) e_{r,k,i}^{(t)}(l)  \ts [l+\de_{ri}^{(k)}] ,\\
f_r^{(t)}\ts [l] &=
\sum_{i,k}\Phi(l,-\de_{ri}^{(k)})f_{r,k,i}^{(t)}(l) \ts [l-\de_{ri}^{(k)}].
 \end{align}

We proceed with the verification of defining relations.


{\bf 1.}
\begin{equation*}
[d_{i}^{\ts(r)},d_{j}^{\ts(s)}][l]=0.
\end{equation*}
The  statement is obvious.

\

{\bf 2.}
\begin{align}\label{relation ef}
[e_{i}^{(r)},f_{j}^{(s)}][l]&=-\ts \de_{ij}\ts\sum_{t=0}^{r+s-1}
d_{i}^{\tss\prime\ts(t)}\ts d_{i+1}^{\ts(r+s-t-1)}[l].
\end{align}

\

The tableaux that appear in the equation \eqref{relation ef} are of  the form $[l+\delta^{(k_1)}_{i,u_1}-\delta^{(k_2)}_{j,u_2}]$.
Assume $[l+\delta^{(k_1)}_{i,u_1}-\delta^{(k_2)}_{j,u_2}]\in \mathcal{B}_{\mathcal{C}}([l])$ and $|i-j|>1$. Under these conditions  $[l+\delta^{(k_1)}_{i,u_1}],[l -\delta^{(k_2)}_{j,u_2}]\in \mathcal{B}_{\mathcal{C}}([l])$.
Let $[v]$ be a tableau with $\mathbb Z$-independent entries.
Then we have $[e_{i}^{(r)},f_{j}^{(s)}][v]=0$.
Therefore the coefficient of $[l+\delta^{(k_1)}_{i,u_1}-\delta^{(k_2)}_{j,u_2}]$ on both sides of \eqref{relation ef} is equal.

Suppose now that  $|i-j|=1$ and there is no relation between $(k_1,i,u_1)$ and $(k_2,j,u_2)$.
Similarly to the case $|i-j|>1$, let $[v]$ be a tableau with $\mathbb Z$-independent entries.
 By comparing the  coefficients of $[l+\delta^{(k_1)}_{i,u_1}-\delta^{(k_2)}_{j,u_2}]$ and  $[v+\delta^{(k_1)}_{i,u_1}-\delta^{(k_2)}_{j,u_2}]$  we conclude that
 the coefficient of $[l+\delta^{(k_1)}_{i,u_1}-\delta^{(k_2)}_{j,u_2}]$ on both sides of \eqref{relation ef}  is equal.

Suppose $|i-j|=1$ and there is a relation between $(k_1,i,u_1)$ and $(k_2,j,u_2)$.
We denote by $\mathcal{C}'$  the set that consists of this relation.
Let $[v]$ be a tableau such that $v_{i,u_m}^{(k_m)}=l_{i,u_m}^{(k_m)}$, $m=1,2$ and all other entries are $\mathbb Z$-independent. By Example \ref{example admissible} $V_{\mathcal{C}'}([v])$ is a $W(\pi)$-module.
Thus $[e_{i}^{(r)},f_{j}^{(s)}][v]=0$.
Since $[l+z]\in  \mathcal{B}_{\mathcal{C}}([l])$ if and only if $[v+z]\in  \mathcal{B}_{\mathcal{C}'}([v])$
where $z=\delta^{(k_1)}_{i,u_1},-\delta^{(k_2)}_{j,u_2}$ or $  \delta^{(k_1)}_{i,u_1}-\delta^{(k_2)}_{j,u_2}$.
Therefore the coefficient of $[l+\delta^{(k_1)}_{i,u_1}-\delta^{(k_2)}_{j,u_2}]$ on both sides of \eqref{relation ef} is equal.

Suppose $i=j$ and $(k_1,u_1)\neq(k_2,u_2)$. Then there is no relation between
$(k_1,i,u_1)$ and $(k_2,i,u_2)$. Similarly to the case $|i-j|>1$, we prove that
the coefficient of $[l+\delta^{(k_1)}_{i,u_1}-\delta^{(k_2)}_{j,u_2}]$ on both sides of \eqref{relation ef} are equal.

Suppose $i=j$ and $(k_1,u_1)=(k_2,u_2)=(k,u)$.
Let $[v]$ be a tableau with $\mathbb Z$-independent entries.
Then $[e_{i}^{(r)},f_{j}^{(s)}][v] =-\ts \de_{ij}\ts\sum_{t=0}^{r+s-1}
d_{i}^{\tss\prime\ts(t)}\ts d_{i+1}^{\ts(r+s-t-1)}[v]$.

The coefficient of $[l]$ on the left hand side is as follows:
\begin{align*}
\sum_{k,u}  e_{i,k,u}^{(r)}(v-\de_{iu}^{(k)})f_{i,k,u}^{(s)}(v)
-\sum_{k,u}  f_{i,k,u}^{(s)}(v+\de_{iu}^{(k)})e_{i,k,u}^{(r)}(v).
\end{align*}
We denote the coefficient of $[v]$ on the right hand side by $h(v)$.

Since $\Phi(l,-\de_{iu}^{(k)},\de_{iu}^{(k)})=\Phi(l,-\de_{iu}^{(k)})$ and
$\Phi(l, \de_{iu}^{(k)},-\de_{iu}^{(k)})=\Phi(l, \de_{iu}^{(k)})$,
the coefficient of $[l]$ in $[e_{i}^{(r)},f_{j}^{(s)}][l]$ is
\begin{align*}
\sum_{k,u}  \Phi(l,-\de_{iu}^{(k)} )e_{i,k,u}^{(r)}(l-\de_{iu}^{(k)})f_{i,k,u}^{(s)}(l)
-\sum_{k,u} \Phi(l,\de_{iu}^{(k)} ) f_{i,k,u}^{(s)}(l+\de_{iu}^{(k)})e_{i,k,u}^{(r)}(l).
\end{align*}

The coefficient of $[l]$ in $-\ts \de_{ij}\ts\sum_{t=0}^{r+s-1}
d_{i}^{\tss\prime\ts(t)}\ts d_{i+1}^{\ts(r+s-t-1)}[l]$ is $h(l)$. We have

\begin{equation*}
\begin{split}
\sum_{k,u}  \Phi(l,-\de_{iu}^{(k)} )e_{i,k,u}^{(r)}(l-\de_{iu}^{(k)})f_{i,k,u}^{(s)}(l)
-\sum_{k,u} \Phi(l,\de_{iu}^{(k)} ) f_{i,k,u}^{(s)}(l+\de_{iu}^{(k)})e_{i,k,u}^{(r)}(l)\\
=\sum_{k,u, \Phi(l,-\de_{iu}^{(k)})=1 }e_{i,k,u}^{(r)}(l-\de_{iu}^{(k)})f_{i,k,u}^{(s)}(l)
-\sum_{k,u, \Phi(l,\de_{iu}^{(k)} )=1} f_{i,k,u}^{(s)}(l+\de_{iu}^{(k)})e_{i,k,u}^{(r)}(l)\\
=\lim_{v\rightarrow l}\left(\sum_{k,u, \Phi(l,-\de_{iu}^{(k)})=1 }e_{i,k,u}^{(r)}(v-\de_{iu}^{(k)})f_{i,k,u}^{(s)}(v)
-\sum_{k,u, \Phi(l,\de_{iu}^{(k)} )=1} f_{i,k,u}^{(s)}(v+\de_{iu}^{(k)})e_{i,k,u}^{(r)}(v)\right).
\end{split}
\end{equation*}

In order to show

\begin{align*}
\sum_{k,u}  \Phi(l,-\de_{iu}^{(k)} )e_{i,k,u}^{(r)}(l-\de_{iu}^{(k)})f_{i,k,u}^{(s)}(l)
-\sum_{k,u} \Phi(l,\de_{iu}^{(k)} ) f_{i,k,u}^{(s)}(l+\de_{iu}^{(k)})e_{i,k,u}^{(r)}(l)=h(l)
\end{align*}
it is sufficient to prove that
\begin{align}\label{eq-lim}
\lim_{v\rightarrow l}\left(\sum_{k,u, \Phi(l,-\de_{iu}^{(k)})=0 }e_{i,k,u}^{(r)}(v-\de_{iu}^{(k)})f_{i,k,u}^{(s)}(v)
-\sum_{k,u, \Phi(l,\de_{iu}^{(k)} )=0} f_{i,k,u}^{(s)}(v+\de_{iu}^{(k)})e_{i,k,u}^{(r)}(v)\right)=0.
\end{align}

computing we have

\begin{align*}
\lim_{v\rightarrow l}\left(\sum_{k,u, \Phi(l,-\de_{iu}^{(k)})\neq 0 }e_{i,k,u}^{(r)}(v-\de_{iu}^{(k)})f_{i,k,u}^{(s)}(v)
-\sum_{k,u, \Phi(l,\de_{iu}^{(k)} )\neq 0} f_{i,k,u}^{(s)}(v+\de_{iu}^{(k)})e_{i,k,u}^{(r)}(v)\right) \\
=\sum_{k,u}  \Phi(l,-\de_{iu}^{(k)} )e_{i,k,u}^{(r)}(l-\de_{iu}^{(k)})f_{i,k,u}^{(s)}(l)
-\sum_{k,u} \Phi(l,\de_{iu}^{(k)} ) f_{i,k,u}^{(s)}(l+\de_{iu}^{(k)})e_{i,k,u}^{(r)}(l),
\end{align*}

\begin{align*}
\lim_{v\rightarrow l}\left(\sum_{k,u }e_{i,k,u}^{(r)}(v-\de_{iu}^{(k)})f_{i,k,u}^{(s)}(v)
-\sum_{k,u } f_{i,k,u}^{(s)}(v+\de_{iu}^{(k)})e_{i,k,u}^{(r)}(v)\right)  \\
=\sum_{k,u}  \Phi(l,-\de_{iu}^{(k)} )e_{i,k,u}^{(r)}(l-\de_{iu}^{(k)})f_{i,k,u}^{(s)}(l)
-\sum_{k,u} \Phi(l,\de_{iu}^{(k)} ) f_{i,k,u}^{(s)}(l+\de_{iu}^{(k)})e_{i,k,u}^{(r)}(l).
\end{align*}

On the other hand,
\begin{align*}
 \left(\sum_{k,u }e_{i,k,u}^{(r)}(v-\de_{iu}^{(k)})f_{i,k,u}^{(s)}(v)
-\sum_{k,u } f_{i,k,u}^{(s)}(v+\de_{iu}^{(k)})e_{i,k,u}^{(r)}(v)\right)
=h(v)
\end{align*}
where $h(v)$ is a polynomial in $v$. Then the limit is $h(l)$.

Thus we have
\begin{align*}
\sum_{k,u}  \Phi(l,-\de_{iu}^{(k)} )e_{i,k,u}^{(r)}(l-\de_{iu}^{(k)})f_{i,k,u}^{(s)}(l)
-\sum_{k,u} \Phi(l,\de_{iu}^{(k)} ) f_{i,k,u}^{(s)}(l+\de_{iu}^{(k)})e_{i,k,u}^{(r)}(l)=h(l).
\end{align*}

The following statements can be verified by direct computation:
\begin{itemize}
\item[(i)] If $\Phi(l,-\de_{iu}^{(k)})=0$  and
$l_{iu}^{k}-l_{iu'}^{k'}\neq 1$ for any $(k',u')\neq(k,u)$, then
$$\lim\limits_{v\rightarrow l}e_{i,k,u}^{(r)}(v-\de_{iu}^{(k)})f_{i,k,u}^{(s)}(v)=0.$$
\item[(ii)] If $\Phi(l,\de_{iu}^{(k)} )=0$ and
$l_{iu'}^{k'}-l_{iu}^{k}\neq 1$ for any $(k',u')\neq(k,u)$,
 then $$\lim\limits_{v\rightarrow l}f_{i,k,u}^{(s)}(v+\de_{iu}^{(k)})e_{i,k,u}^{(r)}(v)=0.$$
\item[(iii)] If $l_{iu}^{k}-l_{iu'}^{k'}=1$,
then $\Phi(l,-\de_{iu}^{(k)})= \Phi(l,\de_{iu'}^{(k')} )=0 $ and

$$\lim\limits_{v\rightarrow l}
\left(
e_{i,k,u}^{(r)}(v-\de_{iu}^{(k)})f_{i,k,u}^{(s)}(v)-f_{i,k',u'}^{(s)}(v+\de_{iu}^{(k)})e_{i,k',u'}^{(r)}(v)
\right)
=0.$$
\end{itemize}

Therefore \eqref{eq-lim} holds
and
we complete the proof.

\

{\bf 3.}
\begin{align}\label{relation de}
[d_{i}^{\ts(r)},e_{j}^{(s)}][l]&=(\de_{ij}-\de_{i,j+1})\ts\sum_{t=0}^{r-1}
d_{i}^{\ts(t)}\ts e_{j}^{(r+s-t-1)}[l], \\ \label{relation df}
[d_{i}^{\ts(r)},f_{j}^{(s)}][l]&=
(\de_{i,j+1}-\de_{ij})\ts\sum_{t=0}^{r-1}
f_{j}^{(r+s-t-1)}\ts d_{i}^{\ts(t)}[l].
\end{align}

\

To prove that for every $[l+\delta^{(k)}_{j,t}]\in \mathcal{B}_{\mathcal{C}}([l])$,
the coefficients on both sides of \eqref{relation de} are equal,  consider a tableau $[v]$ with $\mathbb Z$-independent entries. For $[v]$ the coefficients on both sides of \eqref{relation de} are equal. Taking the limit $v\rightarrow l$ we obtain the statement.

The Relation (\ref{relation df}) can be proved by the same argument.



\



{\bf 4.} \label{lemm ee}

\begin{align}\label{relation ee}
[e_{i}^{(r)},e_{i}^{(s+1)}][l]-[e_{i}^{(r+1)},e_{i}^{(s)}][l]&=
e_{i}^{(r)}e_{i}^{(s)}[l]+e_{i}^{(s)}e_{i}^{(r)}[l],
\\  \label{relation ff}
[f_{i}^{(r+1)},f_{i}^{(s)}][l]-[f_{i}^{(r)},f_{i}^{(s+1)}][l]&=
f_{i}^{(r)} f_{i}^{(s)}[l]+f_{i}^{(s)} f_{i}^{(r)}[l].
\end{align}

\

The tableaux which appear in the Equation \eqref{relation ee} are of the form $[l+2\delta^{(k)}_{i,s}]$ and
$[l+ \delta^{(k)}_{i,s}+\delta^{(r)}_{i,t}]$, $(k,s)\neq(r,t)$.
In the following we show the for any such tableau in  $\mathcal{B}_{\mathcal{C}}([l])$
the coefficients on both sides of \eqref{relation ee} are equal.
It easy to see that when $[l+2\delta^{(k)}_{i,s}]\in \mathcal{B}_{\mathcal{C}}([l])$ then
$[l+\delta^{(k)}_{i,s}]\in \mathcal{B}_{\mathcal{C}}([l])$.
Hence the corresponding value of $\Phi$ is $1$ and the  coefficients on both sides are equal.
Similarly, if
$[l+ \delta^{(k)}_{i,s}+\delta^{(r)}_{i,t}]\in \mathcal{B}_{\mathcal{C}}([l])$ then
$[l+ \delta^{(k)}_{i,s}] ,[l+  \delta^{(r)}_{i,t}]\in \mathcal{B}_{\mathcal{C}}([l])$.
Thus the coefficients of $[l+ \delta^{(k)}_{i,s}+\delta^{(r)}_{i,t}]$, $(k,s)\neq(r,t)$
on both sides of \eqref{relation ee} are equal.

Consider a tableau $[v]$ with $\mathbb Z$-independent entries. For $[v]$ the coefficients on both sides of \eqref{relation ee} are equal.
Taking the limit $v\rightarrow l$ we obtain Equation \eqref{relation ee}.

The Relation (\ref{relation ff}) can be proved using the same arguments.

\

{\bf 5.}
\begin{align}\label{relation ee1}
[e_{i}^{(r)},e_{i+1}^{(s+1)}][l]-[e_{i}^{(r+1)},e_{i+1}^{(s)}][l]&=
-e_{i}^{(r)}e_{i+1}^{(s)}[l],\\ \label{relation ff1}
[f_{i}^{(r+1)},f_{i+1}^{(s)}][l]-[f_{i}^{(r)},f_{i+1}^{(s+1)}][l]&=
-f_{i+1}^{(s)}f_{i}^{(r)}[l].
\end{align}

\

The tableaux which appear in Equation \eqref{relation ee1} are  of the form
$[l+ \delta^{(k)}_{i,s}+\delta^{(r)}_{i+1,t}]$.
Let $[l+ \delta^{(k)}_{i,s}+\delta^{(r)}_{i+1,t}]\in \mathcal{B}_{\mathcal{C}}([l])$.
If there is no relation between $(k,i,s)$ and $(r,i+1,t)$,
then $[l+ \delta^{(k)}_{i,s} ], [l +\delta^{(r)}_{i+1,t}]\in \mathcal{B}_{\mathcal{C}}([l])$.
By the argument on the proof of ({\bf  \ref{lemm ee}}) we have the same coefficients of $[l+ \delta^{(k)}_{i,s}+\delta^{(r)}_{i+1,t}]$ on both sides of \eqref{relation ee1}.

Assume $\mathcal{C}'=\{(r,i+1,t)\geq(k,i,s)\}\subset \mathcal{C}$. It is admissible by Example \ref{example admissible}.
Let $[v]$ be a tableau such that $v_{i+1,t}^{(r)}=l_{i+1,t}^{(r)}$, $v_{i,s}^{(k)}=l_{i,s}^{(k)}$ and all other entries are $\mathbb Z$-independent.
Then $V_{\mathcal{C}'}([v])$ is a $W(\pi)$-module and
$
[e_{i}^{(r)},e_{i+1}^{(s+1)}][v]-[e_{i}^{(r+1)},e_{i+1}^{(s)}][v]=
-e_{i}^{(r)}e_{i+1}^{(s)}[v]
$.
Since $[l+ z] \in \mathcal{B}_{\mathcal{C}}([l])$ if and only if $[v+ z] \in \mathcal{B}_{\mathcal{C}'}([v])$ for
$z=\delta^{(k)}_{i,s},\delta^{(r)}_{i+1,t}$,
by substituting $l$ for $v$ in the
coefficients of $[v+ \delta^{(k)}_{i,s}+\delta^{(r)}_{i+1,t}]$  we obtain
the coefficients of $[l+ \delta^{(k)}_{i,s}+\delta^{(r)}_{i+1,t}]$. Therefore
the coefficient of $[l+ \delta^{(k)}_{i,s}+\delta^{(r)}_{i+1,t}]$ on both sides of \eqref{relation ee1} are equal.

Similarly one treats the case when  $\{(k,i,s)>(r,i+1,t)\}\subset \mathcal{C}$.
This completes the proof of  \eqref{relation ee1}.
The equality \eqref{relation ff1} can be proved by the same argument.

\

{\bf 6.}
\begin{align*}
[e_{i}^{(r)},e_{j}^{(s)}][l]&=0,\qquad&&\text{if}\quad |i-j|>1,\\
[f_{i}^{(r)},f_{j}^{(s)}][l]&=0,\qquad&&\text{if}\quad |i-j|>1.
\end{align*}

\


The proof is analogous to the proof of Relations (\ref{relation ee1}) and (\ref{relation ff1}).

\

{\bf 7.}
\begin{align}\label{relation eee}
[e_{i}^{(r)},[e_{i}^{(s)},e_{j}^{(t)}]][l]
&+[e_{i}^{(s)},[e_{i}^{(r)},e_{j}^{(t)}]][l]=0,
\qquad&&\text{if}\quad |i-j|=1, \\
[f_{i}^{(r)},[f_{i}^{(s)},f_{j}^{(t)}]][l]
&+[f_{i}^{(s)},[f_{i}^{(r)},f_{j}^{(t)}]][l]=0,
\qquad&&\text{if}\quad |i-j|=1.
\end{align}

\

Potential tableaux  in the equality \eqref{relation eee} are of the form
$[l+ \delta^{(k_1)}_{i,u_1}+\delta^{(k_2)}_{i,u_2}+\delta^{(k_3)}_{i,u_3}]$ (we want to show that the coefficient of such tableaux is zero).
Assume $[l+ \delta^{(k_1)}_{i,u_1}+\delta^{(k_2)}_{i,u_2}+\delta^{(k_3)}_{i,u_3}]\in  \mathcal{B}_{\mathcal{C}}([l])$.
Suppose first that there is no relation between  $(k_1,i,u_1)$, $(k_2,i,u_2)$ and $(k_3,j,u_3)$.
Let $[v]$ be a tableau with $\mathbb Z$-independent  entries.
Then $$[e_{i}^{(r)},[e_{i}^{(s)},e_{j}^{(t)}]][v]
+[e_{i}^{(s)},[e_{i}^{(r)},e_{j}^{(t)}]][v]=0.$$
The coefficient of $[l+ \delta^{(k_1)}_{i,u_1}+\delta^{(k_2)}_{i,u_2}+\delta^{(k_3)}_{j,u_3}]$ is obtained by substituting
$l$ for $v$ in the coefficient of $[v+ \delta^{(k_1)}_{i,u_1}+\delta^{(k_2)}_{i,u_2}+\delta^{(k_3)}_{j,u_3}]$, which is zero.

Suppose $j=i+1$.
If $(k_3,j,u_3)\geq (k_1,i,u_1)$ and there is no relation between $(k_3,j,u_3)$ and $ (k_2,i,u_2)$, then set
 $\mathcal{C}'=\{ (k_3,j,u_3)\geq (k_1,i,u_1)\}$. Consider a tableau
$[v]$  such that $v_{i_m,u_m}^{(k_m)}=l_{i_m,u_m}^{(k_m)}$ for
$m=1,2,3$, $i_1=i_2=i,i_3=i+1$ and all other entries are $\mathbb Z$-independent.
Then $[l+z]\in  \mathcal{B}_{\mathcal{C}}([l])$ if and only if $[v+z]\in  \mathcal{B}_{\mathcal{C}'}([v])$,
where $z$ is $\delta_{i_m,u_m}^{(k_m)}$ or $\delta_{i_{m_1},u_{m_1}}^{(k_{m_1})}+\delta_{i_{m_2},u_{m_2}}^{(k_{m_2})}$, $m,m_1,m_2=1,2,3$.

Since $[e_{i}^{(r)},[e_{i}^{(s)},e_{j}^{(t)}]][v]
+[e_{i}^{(s)},[e_{i}^{(r)},e_{j}^{(t)}]][v]=0$, we have that the coefficient  of
$[l+ \delta^{(k_1)}_{i,u_1}+\delta^{(k_2)}_{i,u_2}+\delta^{(k_3)}_{j,u_3}]$  is zero.

If there exists a relation between $(k_3,j,u_3)$ and $ (k_1,i,u_1)$, then by the same argument one can show that the coefficient  of
$[l+ \delta^{(k_1)}_{i,u_1}+\delta^{(k_2)}_{i,u_2}+\delta^{(k_3)}_{j,u_3}]$  is zero.

If $(k_1,i,u_1)>(k_3,j,u_3)$ and $(k_3,j,u_3)\geq (k_2,i,u_2)$, then there exits
$(k_4,i-1,u_4)$ such that $(k_1,i,u_1)\geq(k_4,i-1,u_4)$ and $(k_4,i-1,u_4)\geq (k_2,i,u_2)$.
Let $\mathcal{C}'=\{(k_1,i,u_1)>(k_3,j,u_3),(k_3,j,u_3)\geq(k_2,i,u_2),
(k_1,i,u_1)\geq(k_4,i-1,u_4),(k_4,i-1,u_4)> (k_2,i,u_2)\}$ and
$[v]$  a tableau such that $v_{i_m,u_m}^{(k_m)}=l_{i_m,u_m}^{(k_m)}$ for
$m=1,2,3,4$, $i_1=i_2=i,i_3=i+1,i_4=i-1$ and all other entries are $\mathbb Z$-independent.
Then $[l+z]\in  \mathcal{B}_{\mathcal{C}}([l])$ if and only if $[v+z]\in  \mathcal{B}_{\mathcal{C}'}([v])$,
where $z$ is $\delta_{i_m,u_m}^{(k_m)}$ or $\delta_{i_{m_1},u_{m_1}}^{(k_{m_1})}+\delta_{i_{m_2},u_{m_2}}^{(k_{m_2})}$,$m,m_1,m_2=1,2,3$.

Since $[e_{i}^{(r)},[e_{i}^{(s)},e_{j}^{(t)}]][v]
+[e_{i}^{(s)},[e_{i}^{(r)},e_{j}^{(t)}]][v]=0$, we have that the coefficient  of
$[l+ \delta^{(k_1)}_{i,u_1}+\delta^{(k_2)}_{i,u_2}+\delta^{(k_3)}_{j,u_3}]$  is zero.

The case $j=i-1$ is treated similarly.
This completes the proof of \eqref{relation eee}.
The second equality can be proved in the same way. We complete the proof of the
 sufficiency of conditions in  Theorem \ref{sufficiency of admissible}.

Suppose there exists an adjoining triple $(k,i,j),(r,i,t)$ which does not satisfy the condition \eqref{condition for admissible}. Then applying the RR-method to $\mathcal{C}$, after finitely many steps,  we will obtain a set of relations from Example
\ref{example non admissilbe} which is not admissible. Thus $\mathcal{C}$ is not admissible.

\section*{Acknowledgment}
V.F. is
supported in part by the CNPq grant (200783/2018-1) and by the
Fapesp grant (2014/09310-5).
J. Z. is supported by the Fapesp grant (2015/05927-0).

\end{document}